\documentclass[12pt]{article}
\usepackage{caption}
\usepackage{url}
\usepackage{color}
\usepackage{amssymb}
\usepackage{amsmath}
\usepackage{amsthm}
\usepackage{graphicx}

\newtheorem{theorem}{Theorem}[section]
\newtheorem{lemma}[theorem]{Lemma}

\newtheorem{definition}[theorem]{Definition}

\newtheorem{corollary}[theorem]{Corollary}
\newtheorem{claim}[theorem]{Claim}
\newtheorem{remark}[theorem]{Remark}
\newtheorem{conjecture}[theorem]{Conjecture}
\newtheorem{proposition}[theorem]{Proposition}
\newtheorem{open}[theorem]{Open Question}
\newtheorem{example}[theorem]{Example}

\DeclareMathOperator{\diam}{diam}
\DeclareMathOperator{\var}{Var}
\DeclareMathOperator{\cov}{Cov}
\DeclareMathOperator{\modd}{mod}

\title{Probabilistic and Geometrical Applications to Graph Theory}
\author{Matthew P. Yancey \thanks{Institute for Defense Analyses / Center for Computing Sciences (IDA / CCS), mpyance@super.org}}

\begin{document}
\maketitle

\begin{abstract}
This paper consists of two halves.

In the first half of the paper, we consider real-valued functions $f$ whose domain is the vertex set of a graph $G$ and that are Lipschitz with respect to the graph distance.
By placing a uniform distribution on the vertex set, we treat $f$ as a random variable.
We investigate the link between the isoperimetric function of $G$ and the functions $f$ that have maximum variance or meet the bound established by the subgaussian inequality.
We present several results describing the extremal functions, and use those results to resolve: (A) a conjecture by Bobkov, Houdr\'e, and Tetali characterizing the extremal functions of the subgaussian inequality of the odd cycle, and (B) a conjecture by  Alon, Boppana, and Spencer on the relationship between maximum variance functions and the isoperimetric function of product graphs.

While establishing a discrete analogue of the curved Brunn-Minkowski inequality for the discrete hypercube, Ollivier and Villani suggested several avenues for research.
We resolve them in second half of the paper as follows.
\begin{itemize}
	\item They propose that a bound on $t$-midpoints can be obtained by repeated application of the bound on midpoints, if the original sets are convex.  We construct a specific example where this reasoning fails, and then prove our construction is general by characterizing the convex sets in the discrete hypercube.
	\item A second proposed technique to bound $t$-midpoints involves new results in concentration of measure.  We follow through on this proposal, with heavy use on results from the first half of the paper.
	\item We show that the curvature of the discrete hypercube is not positive or zero.
\end{itemize}

\end{abstract}

\section{Motivation}

This manuscript deals with graphs, and the attempts to apply alternative areas of mathematics to them.

The first half is motivated by concentration of measure.
There is a canonical method to construct a Martingle by iteratively selecting a random variable $X$ that is defined as a function that is Lipschitz over a graph.
The isoperimetric function of the graph has been linked to the extremal variables.
We will present results describing the extremal variables, which will allow us to refine our knowledge about this link.

The second half is motivated by geometry.
We recently investigated the relationship between negative curvature and congestion in transportation networks \cite{Y}.
We currently are interested in showing that networks exhibiting qualities associated with positively curved spaces will consequently have many routing options for transportation between locations.
In this paper, we use an abundance of midpoints as a proxy for an abundance of routing options.
Our previous work had the advantage of a discrete analogue of negative curvature (Gromov's $4$-points hyperbolicity) that is well-studied \cite{Bow,ABCFLMSS,G}, practical \cite{CCL,CDEHV,CDEHV2,CDEHVX}, and consequential \cite{ST_hyper,JL,J,ADM}.
Multiple discrete analogues of positive curvature have been proposed \cite{EM,O1,O2,B,BS,GRST}, and some consequences of those notions are known \cite{P,CP,EM,BCLL}.
We will present results about the discrete analogue of the curved Brunn-Minkowski inequality.

\subsection{Background: concentration of measure}
For graphs $G_1, \ldots, G_k$, the Cartesian product $G_1 \square \cdots \square G_k$ is the graph with vertex set $V_1 \times \cdots \times V_k$ such that the distance between vertices $v = (v_1, \ldots, v_k)$ and $u = (u_1, \ldots, u_k)$ is $\sum_i d_{G_i}(v_i, u_i)$.
The edges of $G_1 \square \cdots \square G_k$ are the vertices that are distance $1$ apart.
We denote the Cartesian product with $G=G_1=\cdots=G_k$ as $G^k$.

For a fixed graph $G$ and vertex set $S$, let $B_d(S) = \{u : d_G(u,S) \leq d\}$.
The \emph{isoperimetric function} is $i_{G,d} = \min_{|S| \geq |V(G)|/2} |B_d(S)|$.
A problem considered by several authors \cite{ABS,BHT,ST} is the isoperimetric function of product spaces.
That is, we generalize the isoperimetric function from $G$ to $G^n$ as $i_{G,d,n} = \min_{|S| \geq |V(G^n)|/2} |B_d(S)|$.

One method to analyze the isoperimetric function of product spaces is to study probability spaces defined as uniform probabilities over the vertex set of a graph $G$ equipped with functions $X:V(G) \rightarrow \mathbb{R}$ that are Lipschitz with respect to the standard graph distance in $G$.
By placing a uniform distribution on $V(G)$, we abuse notation and treat $X$ as a random variable.
We use the notation from \cite{BHT} that $c(G) = \max_X \sqrt{\var(X)}$, where $X$ is taken over Lipschitz functions on $V(G)$.
This notation is not consistent with \cite{ABS,ST}.
We will call a function $X$ \emph{variance-optimal} if $X$ is Lipschitz and $\var(X) = c^2(G)$.
Alon, Boppana, and Spencer \cite{ABS} proved that $V(G^n) - i_{G,d,n}$ decays exponentially as $d$ grows when $\sqrt{n} \ll d \ll n$ with a rate that relies on $c^2(G)$.
Let $m(X)$ be the median value of $X$, and note that if $X$ is variance-optimal, then so is $-X$ and $X+a$ for all $a \in \mathbb{R}$.

\begin{theorem}[\cite{ABS}] \label{variance optimal is mostly right}
Let $\sqrt{n} \ll d \ll n$.
We have that $i_{G,d,n}/|V(G^n)| > 1 - e^{-\frac{d^2}{2c^2(G)n}(1+o(1))}$.
Let $S_{r,X} = \{a = (a_1, a_2, \ldots, a_n) \in V(G^n) : X^n(a) = \sum_{i=1}^n X(a_i) \leq r\}$.
If $X$ is variance-optimal and $m(X^n) \leq r \leq \mathbb{E}(X) = 0$, then $|S_{r,X}| \geq \frac{1}{2}|V(G^n)|$ and $|B_d(S_{r,X})|/|V(G^n)| \leq 1 - e^{-\frac{d^2}{2c^2(G)n}(1+o(1))}$.
\end{theorem} 

They also conjectured a stronger relationship between $i_{G,d,n}$ and $c^2(G)$---that the extremal set is determined by variance-optimal functions.

\begin{conjecture}[\cite{ABS}, page 416]\label{is variance optimal correct}
Let $X$ be a variance-optimal function over $G$.
Is it true for $n$ sufficiently large and $d,r$ in appropriate ranges that $|B_d(S_r)| \leq |B_d(S')|$ for all $S' \subset V(G^n)$ with $|S_r| \leq |S'|$?
\end{conjecture}

Our initial intuition was that the conjecture may be true.
A variable that maximizes variance will attempt to evenly spread values towards $\infty$ and $-\infty$ as possible.
Because $X$ is Lipschitz we have that $B_d(S_r) \subseteq S_{r + d}$.
The hope is that $V(G^n) - S_{r + d,X} = S_{-(r+d),-X}$ is relatively large.

If true, the conjecture would have significance. 
Let $X_1,X_2$ be variance-optimal over $G_1,G_2$ respectively.
Let $H = G_1 \square G_2$, and define variable $Y$ over $H$ as $Y(u_1, u_2) = X_1(u_1) + X_2(u_2)$.
Under these conditions, $Y$ is variance-optimal and $c^2(H) = c^2(G_1) + c^2(G_2)$.
If true, the conjecture would thus imply that the isoperimetric function of $G^n$ and the associated extremal vertex sets can be determined by the isoperimetric function of $G$ and the associated extremal vertex sets.

The conjecture has been proven true for the discrete hypercube by Harper \cite{H}, the Euclidean lattice by Bollob\'{a}s and Leader \cite{BL1}, and the discrete torus by Bollob\'{a}s and Leader \cite{BL2}.

Our initial set of results are a series of statements that describe the variance-optimal functions of a graph.
In Section \ref{variance-optimal and isoperimetric section} we characterize the set of variance-optimal functions for three families of graphs.
Each of those families consists of trees with long paths, and our most useful statement involves hairs.
A \emph{hair} of $G$ is a sequence of vertices $w_0, w_1, \ldots, w_k$ such that $N(w_i) = \{w_{i-1}, w_{i+1}\}$ for $1 \leq i \leq k-1$ and $N(w_k) = \{w_{k-1}\}$.

\textbf{ Lemma \ref{hairs go one way}, (\ref{unimodular hairs}), Remark \ref{hair matched by path}, and Remark \ref{variance too}} \textit{
Let $X$ be a variance-optimal function over $G$ with hair $w_0, w_1, \ldots, w_k$.
The sequence of values $X(w_0), X(w_1), \ldots, X(w_k)$ is unimodular.
Let $m = \min_i X(w_i)$ and $M = \max_i X(w_i)$.
If $G$ has vertices $u', u''$ such that $X(u') \leq m+1$, $X(u'') \geq M-1$, and $u',u'' \notin \{w_1, \ldots, w_k\}$, then the sequence $X(w_0), X(w_1), \ldots, X(w_k)$ is monotone. 
}

We compare this to the structure result of Alon, Boppana, and Spencer.

\begin{theorem}[\cite{ABS}] \label{defining the origin}
Let $X$ be a variance-optimal function over $G$, and define $\mathbb{O}(X) = \{v \in G: -0.5 < X(v) - \mathbb{E}(x) \leq 0.5\}$.
Let $\nu_X = X(u) - \mathbb{E}(X)$ for some $u \in \mathbb{O}(X)$, and let $C_1, \ldots, C_k$ be the connected components of $G - \mathbb{O}(X)$.
Under these conditions, there exists variables $(\delta_1, \ldots, \delta_k) \in \{-1,1\}^k$ such that for $u \in C_i$ we have that $X(u) - \nu_X = \delta_k d(u, \mathbb{O}(X))$.
\end{theorem} 

The first two examples in Section \ref{variance-optimal and isoperimetric section} are an exploration of the assumption $m(x) \leq \mathbb{E}(X)$ in Theorem \ref{variance optimal is mostly right} that is absent from Conjecture \ref{is variance optimal correct}.
We determine a bound on the isoperimetric number of the third example, which leads to the following result.

\textbf{Theorem \ref{variance optimal is not isoperimetric}} \textit{
Conjecture \ref{is variance optimal correct} is not true.
}

We conjecture a simple characterization of variance-optimal functions over trees; we are proposing that $\mathbb{O}(X)$ is not the correct center of the tree conceptually.
If true, the statement would strengthen Lemma \ref{hairs go one way} significantly.

\begin{conjecture}
Let $X$ be a variance-optimal function over a tree $T$.
There exists a vertex $r$ such that for all vertices $u \in V(T)$, we have that $|X(u) - X(r)| = d(u,r)$.
Moreover, if $T$ is not a path, then we may choose $r$ such that $d(r) \geq 3$.
\end{conjecture}

The range $\sqrt{n} \ll d \ll n$ is necessary in Theorem \ref{variance optimal is mostly right}, as Alon, Boppana, and Spencer work with a different tool from probability when $d \approx \epsilon n$ for some constant $\epsilon$.
Specifically, they work with the \emph{subgaussian inequality}, which states that 
\begin{equation}\label{subgaussian ineq}
 \mathbb{E} e^{t(X - \mathbb{E}X)} \leq e^{\sigma^2t^2/2} .
\end{equation}
The \emph{subgaussian constant} for $X$ is a value for $\sigma_X^2$ such that (\ref{subgaussian ineq}) is true for all real $t$.
For vertex set $S$, we define Lipschitz function $X_S(u) = d_G(u,S)$.
In particular, they showed that $|B_d(S)|/|V(G)| \geq 1 - e^{-(d/\sigma - 1)^2/2}$ when $d \geq \sigma$.

For graph $G$, let $\sigma_G$ be the supremum of $\sigma_X$ for functions $X$ that are Lipschitz over $G$.
We call a function $X$ over graph $G$ \emph{optimal} if $\sigma_X = \sigma_G$.
When it is clear, we drop the subscript from $\sigma$.

Using the same construction as above, Alon, Boppana, and Spencer \cite{ABS} showed that $\sigma^2_{G_1 \square G_2} = \sigma^2_{G_1} + \sigma^2_{G_2}$.
Bobkov, Houdr\'e, and Tetali \cite{BHT} showed that if $G = K_{\ell}$ and $\ell$ is even, then $\sigma_G ^2 = 1/4$.
If $\ell$ is odd, then $\sigma_G^{-2} = 2\ell\log\left(\frac{\ell+1}{\ell-1}\right)$.
Exact values for $\sigma$ are also known for paths and cycles of even length \cite{ST}.
Bobkov, Houdr\'e, and Tetali conjectured a characterization of the optimal functions for odd cycles \cite{BHT}, which was repeated by Sammer and Tetali \cite{ST}.

\begin{conjecture}[\cite{BHT}] \label{spread of cycles}
If  $X$ is an optimal variable over the cycle $C_n$, then there exists an $x_0 \in V(C_n)$ such that $|X(v) - X(v_0)| = d(x_0,v)$ for all $v \in V(C_n)$.
\end{conjecture}

We present several structural statements about optimal variables, such as the analogue of Lemma \ref{hairs go one way}.
Theorem \ref{defining the origin} does not hold for optimal variables, but it does hold for ``half'' of the graph.

\textbf{Corollary \ref{subgaussian origin}} \textit{
If $X$ is an optimal function and $X(u) < \mathbb{E}(X)$, then $\nu_X-X(u) = d(u, \mathbb{O}(X))$.
}

Once again we are able to use our statements describing optimal functions to characterize the optimal functions of specific examples.

\textbf{Theorem \ref{odd cycle optimal}} \textit{
Conjecture \ref{spread of cycles} is true.
}

All of our discussion so far generalizes in the obvious way to metric spaces with finite number of elements.
An example is the symmetric group $S_n$ on $n$ elements equipped with the Hamming distance, which is $d(\sigma, \sigma') = | \{i: \sigma(i) \neq \sigma'(i)\} |$.
Bobkov, Houdr\'e, and Tetali \cite{BHT} bounded $\sigma_{S_n}$. 

\begin{theorem}[\cite{BHT}]\label{subgaussian permutations}
Let $S_n$ be the symmetric group on $n$ elements equipped with the Hamming distance $d_H$.
The subgaussian constant for this space satisfies $\frac{1}{16}(n-1) \leq \sigma_{S_n}^2 \leq n-1$.
\end{theorem}

We are interested in Olivier and Villani's \cite{OV} use of concentration inequalities.
Their concentration inequalities are applied to functions whose domain is a \emph{subset} of the hypercube $\{0,1\}^n$ with $\ell_1$ distance, and thus standard statements about the hypercube will not apply.
Fortunately, these subsets are closed under the group action by the symmetric group.
Using this property, concentration inequalities for this unusual domain can be established (with small, but non-trivial work) from concentration inequalities on functions of the symmetric group equipped with Hamming distance.

Thus, we are interested in the bounds in Theorem \ref{subgaussian permutations}.
In Remark \ref{strict inequality} we show that the upper bound is not sharp, although our improvement is negligible.
In Theorem \ref{at least n/4 for S_n} we show that $\sigma_{S_n}^2 > n/4$.

Let $C_r = \{A \in \{0,1\}^n : |A| = r\}$ and $C_{r_1, r_2, \ldots, r_k} = \cup_{i} C_{r_i}$, equipped with a distance metric equal to the order of the symmetric difference.
Ollivier and Villani \cite{OV} established that $\sigma_{C_{n/2}}^2 \leq n-1$ and $\sigma_{C_{(n-1)/2, (n+1)/2}}^2 \leq n$, which we wish to generalize.

\textbf{ Theorem \ref{different level sets} and Theorem \ref{Linearly far from center} } \textit{
For $n \geq 3$, we have that 
$\sigma_{C_{(n-r)/2,(n+r)/2}}^2 < n - 1 + r^2/4$.
Additionally suppose $c \geq 2$, $\frac{c-1}{2(c+1)}n > R > \sqrt{n \ln(\frac{c}{c-1})}$, and $|n/2 - r_i| \leq R$ for all $i$.
Let $X_*$ be a Lipschitz function over $\{0,1\}^n$, and let $X$ be $X_*$ induced on $C_{r_1, \ldots, r_k}$.
$$ \mathbb{P}\left(X - \mathbb{E}(X_*) \geq h \geq 0\right) < e^{(R+3)\ln(c) - \ln(k) -\frac{h^2}{2(n-1)}}. $$
}

\subsection{Background: graph curvature}
For sets $S,T \subset \mathbb{R}^d$ and $c \in \mathbb{R}$, let $S + T = \{s+t: s \in S, t \in R\}$ and $c S = \{cs : s\in S\}$.
Understanding the extremal properties of $S+T$ is a central goal in algebraic combinatorics.
For example, Roth's theorem states that when $A \subset \mathbb{Z}$ and $A$ is disjoint from $(A+A)/2$, then $A$ has density $0$.
Let $m(S, T) = (S + T)/2$.

One method to bound the size of $S + T$ comes from the Brunn-Minkowski inequality, which states that 
$ V(S + T)^{1/d} \geq V(S)^{1/d} + V(T)^{1/d}, $
where $V(S)$ denotes the volume of measurable set $S$.
Using that $V(c S) = c^d V(S)$ and that $(a + b)/2 \geq \sqrt{ab}$ for positive $a,b$, the Brunn-Minkowski inequality transforms into 
$ V(m(S, T)) \geq \sqrt{V(S) V(T)}.$
In other words, the volume of the midpoints between $S$ and $T$ is at least the geometric average of $V(S)$ and $V(T)$.
This statement can be generalized to weighted geometric averages: if $0 < \rho < 1$, define $m_\rho(S,T) = \rho S + (1 - \rho)T$, and we have that (see \cite{B})
$ V(m_\rho(S, T)) \geq V(S)^{\rho} V(T)^{1- \rho} . $

We define the distance between sets $S,T$ to be $d_*(S,T) = \min\{d(s,t):s \in S, t \in T\}$.
For $\mathbb{R}^n$, the value $d_*(S,T)$ has no affect on midpoints, as $S + (\{u\} + T) = \{u\} + (S + T)$.
However, for smooth complete manifolds with positive curvature $K$, the Brunn-Minkowski inequality strengthens (see \cite{OV}) exponentially as $d_*(S,T)$ grows:
$ V(m(S,T)) \geq \sqrt{V(S) V(T)} e^{\frac{K}{8}d_*(S,T)^2}.$
Should a value $K$ hold for some space, we will call the supremum of such values the \emph{Brunn-Minkowski curvature}.

There have been several attempts to generalize the Brunn-Minkowski inequality to discrete spaces such as $\mathbb{Z}^d$.
These efforts have met several obstacles.
The most obvious obstacle is that the volume function naturally voids sets $S$ whose dimension is less than the overall space.
The discrete analogue of volume is to count points, but this function does not naturally void sets with smaller dimensionality.
As such, the strongest statement possible is that $|S+T| \geq |S| + |T|$, and this is sharp for any dimension.

Some progress can be made when we force the sets $S$ and $T$ to live in higher dimensions.
Ruzsa \cite{R} proved that for $S,T \subset \mathbb{Z}^d$ with $\dim(S+T)=d$ and $|S| \geq |T|$, then $|S+T| \geq |S| + d|T| - {d+1 \choose 2}$.
Gardner and Gronchi \cite{GG} proved that for $S,T \subset \mathbb{Z}^d$ with $dim(T)=d$ and $|S| \geq |T|$, then $|S+T|^{1/d} \geq |S|^{1/d} + \left(\frac{|T|-d}{d!}\right)^{1/d}$.

We consider the question posed verbally by Stroock and written down by Ollivier and Villani (see \cite{OV}), ``what is the curvature of the discrete hypercube?''
We consider the more general question of what graphs display the qualities of a curved space.
We limit ourselves here to discrete analogues of positive curvature; see \cite{Y} for a discussion on discrete analogues of negative curvature.
Proposed notions of positive curvature for graphs include coarse Ricci curvature \cite{O1,O2}, Bakry-Emery version of Ricci curvature \cite{BCLL}, dispersion of heat \cite{EM}, displacement convexity with approximate midpoints \cite{B,BS}, displacement convexity with Gaussian midpoints \cite{GRST}, among others.
(We briefly mention that ``displacement convexity with approximate midpoints'' is a misleading name, as it allows for midpoints of quasi-geodesics.)
We will focus on Brunn-Minkowski curvature and displacement convexity as proposed in \cite{OV}.

The \emph{symmetric discrete midpoints} are 
$$\widehat{m}_\rho(a,b) = \{u : d(a,u) = \lfloor \rho d(a,b) \rfloor, d(a,u) + d(u,b) = d(a,b)\} \cup $$
$$\cup \{u : d(a,u) = \lceil (1-\rho) d(a,b) \rceil, d(a,u) + d(u,b) = d(a,b)\},$$
where $\widehat{m}(a,b) = \widehat{m}_{1/2}(a,b)$, and the discrete analogue of Brunn-Minkowski curvature $K$ is that 
$$ \left|\widehat{m}(S,T)\right| \geq \sqrt{|S| |T|} e^{\frac{K}{8}d_*(S,T)^2}.$$
Roughly speaking, displacement convexity implies that the midpoints of an optimal transportation from $A$ to $B$ should satisfy a similar inequality.
There are several definitions of displacement convexity, and we will wait until Section \ref{disp conv sec} to give a formal definition.

The coarse Ricci curvature of the discrete hypercube in $d$ dimensions is easy to calculate as $\frac{1}{2d}$, and Ollivier and Villani \cite{OV} demonstrated that the $d$-dimensional hypercube has Brunn-Minkowski curvature at least $\frac{1}{2d}$.
But it is a coincidence that these are the same number; in Proposition \ref{hypercube can do better} we show that there exists a $\delta > 0$ such that the $d$-dimensional hypercube has Brunn-Minkowski curvature at least $\frac{1}{2d}(1 + \delta)$ when $d$ is large enough.

The discrete hypercube is unique in that it is a product space where the $\ell_p$ topologies are equivalent for all $p \in [0, \infty)$.
With the following result, we show that the Brunn-Minkowski curvature found by Ollivier and Villani is due to the $\ell_0$ topology; the difficulties of establishing a Brunn-Minkowski inequality for the integer lattice gives insight into why the other product measures are insufficient.
We omit the proof to Theorem \ref{quick l0 product spaces result} as the modifications necessary to the proof of Ollivier and Villani's result are contained in the proof to Theorem \ref{l0 metric}, which we will state soon.

\begin{theorem}\label{quick l0 product spaces result}
Any finite product space with dimension $d$ equipped with the $\ell_0$ metric has Brunn-Minkowski curvature at least $\frac{1}{2d}$.
\end{theorem}

Ollivier and Villani's \cite{OV} work stated many open questions.
The largest of which is to understand the curvature of the discrete hypercube with a discrete analogue of displacement convexity.
They also desire a connection between the different discrete analogues of curvature; explicitly stating that ``the relationship between coarse Ricci curvature and displacement convexity of entropy is unclear, and no implication has been proved or disproved in either direction as far as we know.''

The discrete hypercube has been singled out because authorities on the subject \cite{G2,OV} have labeled it as the best candidate for being the first proven example of a positively curved discrete space.
The discrete hypercube was proven by Erbar and Maas \cite{EM} to have positive heat-dispersion-curvature.
Our first result is a rejection of the connection between coarse Ricci curvature and displacement convexity, and we do so on the discrete hypercube, which additionally implies a disagreement between displacement convexity and heat dispersion.

\textbf{Example \ref{negative curvature}} \textit{
There exists sets in the discrete hypercube such that the entropy of the midpoints is less than the entropy of the endpoints.
}

Example \ref{negative curvature} is specifically a reference to standard displacement convexity, which we distinguish from \emph{weak} displacement convexity.
Weak displacement convexity seems quite general, and it seems like it should hold for most spaces that have Brunn-Minkowski curvature (we note that half of our lemmas towards Theorem \ref{almost curved} are not specific to the hypercube).
Such a result would imply positive curvature for more than just the discrete hypercube, as Neeranartvong, Novak, and Sothanaphan \cite{NNS} proved that Brunn-Minkowski curvature also exists in the symmetric group.
While such a result eludes us, we are able to present the following progress towards proving weak displacement convexity of the hypercube.

\textbf{Theorem \ref{almost curved}} \textit{
For probability distributions $\mu_A$, $\mu_B$ over points in the $d$-dimensional discrete hypercube, there exists an optimal transportation $\tau$ such that the distribution $\mu_c$ over the midpoints $\widetilde{m}(\mu_A, \mu_B)$ satisfies 
$$  S(\mu_C) \geq \frac13 S(\mu_A) + \frac13 S(\mu_B) + \frac{2}{5d^3} (W^2(\mu_A, \mu_B))^2 - \frac{2}{3}.$$
}

The other open questions from Ollivier and Villani involve bounding $\widehat{m}_\rho(A,B)$ for $\rho \neq 1/2$.
The first approach suggested is to bound $|\widehat{m}_{1/4}(A,B)|$ by bounding $|\widehat{m}(\widehat{m}(A,B),B)|$.
Such a relationship would only exist in Riemannian spaces for convex $A$ and $B$, and so Ollivier and Villani suggest the condition $\widehat{m}(A,A) \subseteq A$, $\widehat{m}(B,B) \subseteq B$ as the discrete analogue of convexity.
This condition is not sufficient.
To fully explore this line of thinking, we characterize the family of convex sets in the discrete hypercube.

\textbf{Theorem \ref{convex closure is smaller hypercube}} \textit{
If $S$ is a convex subset of $\mathbb{H}_d$, then $S$ is isomorphic to $\mathbb{H}_{d'}$ for $d' \leq d$.
}

Our characterization proves that Example \ref{two convex sets}, which has vertex sets $A,B$ such that $\widehat{m}(\widehat{m}(A,B),B) \supsetneq \widehat{m}_{1/4}(A,B)$, is general.
This characterization also leads to interesting results that highlight how non-intuitive discrete spaces are.

\textbf{Corollary \ref{closure of a ball}} \textit{
The convex closure of any non-trivial ball in a hypercube is the whole space.
}

\textbf{Corollary \ref{closure of subsets}} \textit{
If $A,B$ are nonempty sets of vertices in the hypercube and $C$ is the convex closure of $\widehat{m}(A,B)$, then $A \cup B \subseteq C$.
}

We are able to use our advances in concentration theory to bound $\widehat{m}_\rho(A,B)$ when $|1/2-\rho|$ is small.
Our bound still uses $\sqrt{|S||T|}$ instead of $|S|^\rho|T|^{1-\rho}$, which implies that our technique is not preferable.

\textbf{Theorem \ref{l0 metric}} \textit{
Suppose $S,T \subseteq G = K_{r_1} \square K_{r_2} \square \cdots \square K_{r_d}$ such that $d_*(S,T) \geq \delta d$.
Let $\rho$ be such that $|1/2-\rho| < \epsilon$ and $C = \delta^2 - 16\ln(2)\epsilon > 0$.
Under these conditions, for large $d$ we have that 
$$|\widehat{m}_{\rho}(S,T)| > \sqrt{|S||T|} e^{C d(1/2+o(1))}.$$
}

Bounding $|\widehat{m}(\widehat{m}(A,B),B)|$ is useful for a strong version of displacement convexity.
Graphs that satisfy such a property also satisfy a condition that is similar, but stronger, than being claw-free.
Unfortunately, we characterize this family of graphs in Theorem \ref{char strong curve}, and it is a very small family of spaces.

\subsection{Common theme}
The obvious connection between the two halves of this manuscript is that the second half uses a technical result from the first half.
However, there is a stronger intuitive link.
While most of the space of the second half of the paper is spent discussing ``what does it mean to be curved?'' the intent is to discuss ``who is curved?''
From that perspective, we propose graphs with small variance---the graphs we study in the first half of the paper---as the model for a positively curved space.

Let us quickly survey the evidence for this.
Lipschitz functions frequently appear as minor statements in manuscripts about curved spaces.
For example, in Particular Case 5.4 of \cite{V}, Villani presents the dual Kantorovich problem (which we elaborate on in Remark \ref{constant valued functions}) as maximizing a variance-like parameter over the set of Lipschitz functions; and Lipschitz functions show up elsewhere in the textbook, such as in Remark 6.5.
The heat dissipation approach to curvature from Erbar and Maas \cite{EM} implies a bound on the subgaussian constant.

The spread can be defined $c^2(G) = \frac{1}{|V(G)|}\max_f \sum_{u \in V(G)} f(u)^2$ subject to the constraints $0 = \sum_{u \in V(G)} f(u)$ and for each $vw \in E(G)$ we have that $(f(v) - f(w))^2 \leq 1$.
The Fiedler vector $h$ of the Laplacian is known to satisfy $\frac{|E(G)|}{\lambda} = \max_h \sum_{u \in V(G)} h(u)^2$ subject to the constraints $0 = \sum_{u \in V(G)} h(u)$ and $\sum_{vw \in E(G)}(h(v) - h(w))^2 \leq |E(G)|$, where $\lambda$ is the associated eigenvalue.
Alon and Milman \cite{AM} were the first to observe that concentrated graphs are a generalization of expanders.
Bauer, Chung, Lin, and Liu \cite{BCLL} showed that each of coarse Ricci curvature and Bakry-Emery version of Ricci curvature implies a bounded value for $\lambda$.

Thus the literature is heavy with results that indicate discrete curvature is a tool that can be used to find concentrated discrete spaces.
But we know far more concentrated discrete spaces than curved discrete spaces; so can we use techniques from expanders to study curvature?
We explore the prospects of such research in Section \ref{and now to my field}.
In continuous spaces, the statements also use curvature to imply concentration and not vice-versa (see chapter 22 of \cite{V}), but the examples contained in this manuscript suggest that any applicable discrete curvature will necessarily be weaker than the continuous analogue.


\section{Concentration Of Lipschitz Functions}

\subsection{Background}
The following is basically the details behind the relevant portions of the first page and a half of \cite{BHT}.
We assume all variables are real.

Recall (\ref{subgaussian ineq}).
Because $X$ is real, the exponential function is positive and monotone increasing, and therefore $e^{th}\mathbb{P}(X - \mathbb{E}X \geq h) \leq \mathbb{E} e^{t(X - \mathbb{E}X)}$ when $t$ is nonnegative.
By setting $t = h\sigma^{-2} \geq 0$ we have that (\ref{subgaussian ineq}) implies
\begin{equation} \label{tail distribution}
\mathbb{P}(X - \mathbb{E}X \geq h) \leq e^{-h^2/(2\sigma^2)}, h \geq 0 .
\end{equation}

Let $m(X)$ be the median value of $X$, and so $y = m(X)$ minimizes the function $\mathbb{E}|X - y|$.
In particular, we have that $\mathbb{E}|X - m(X)| \leq \mathbb{E}|X - \mathbb{E}(X)|$.
We apply the Cauchy-Schwartz inequality (let $Y_e$ be the possible events for variable $|X - \mathbb{E}(X)|$ with probability $p_e$, set $a_e = Y_e \sqrt{p_e}$ and $b_e = \sqrt{p_e}$) to say that $\left(\mathbb{E}|X - \mathbb{E}(X)|\right)^2 \leq \var(X)$.
We conclude that 
\begin{equation} \label{median close to average}
 \left|\mathbb{E}(X) - m(X)\right| \leq \mathbb{E}\left|X - m(X)\right| \leq \mathbb{E}\left|X - \mathbb{E}(X)\right| \leq \sqrt{\var(X)}.
\end{equation}
We may apply (\ref{tail distribution}) with $h = h' - \sqrt{\var(X)}$ and use (\ref{median close to average}) to see that 
\begin{equation} \label{skewed tail distribution prior}
\mathbb{P}(X - m(X) \geq h') \leq e^{-\left(h' - \sqrt{\var(X)}\right)^2/(2\sigma^2)},  h' \geq \sqrt{\var(X)}.
\end{equation}
Claim \ref{var less than subgauss} will transform (\ref{skewed tail distribution prior}) into
\begin{equation} \label{skewed tail distribution}
\mathbb{P}(X - m(X) \geq h') \leq e^{-(h'/\sigma - 1)^2/2},  h' \geq \sigma.
\end{equation}

\begin{claim} \label{var less than subgauss}
$\var(X) \leq \sigma^2$
\end{claim}
\begin{proof}
Consider the expression $Y_t = \mathbb{E}\left(2\left(e^{t(X - \mathbb{E}X)}-1\right)t^{-2}\right)$ as $t \rightarrow 0$.
Apply the Taylor sequence of $e^{t(X - \mathbb{E}X)}$ to see that 
$$\mathbb{E}\left(2\left(e^{t(X - \mathbb{E}X)}-1\right)t^{-2}\right) = \mathbb{E}\left( 2t^{-1}(X - \mathbb{E}(x)) + (X - \mathbb{E}(x))^2 + O(t)\right).$$
Because expectation is linear and $\mathbb{E}\left(X - \mathbb{E}(X)\right) = 0$, we have that $\var(X) = \lim_{t \rightarrow 0}Y_t$.
Now apply (\ref{subgaussian ineq}) to see that $Y_t \leq 2\left(e^{\sigma^2t^2/2}-1\right)t^{-2}$.
By the Taylor sequence for $e^{\sigma^2t^2/2}$ we have that $\lim_{t \rightarrow 0}2\left(e^{\sigma^2t^2/2}-1\right)t^{-2} = \sigma^2$.
\end{proof}

Bobkov, Houdr\'e, and Tetali \cite{BHT} had a special interest in random variables $X_a = f(a)$, where $f$ is Lipschitz over some space $S$.
In particular, they were interested in bounding the size of the space $A^{-h} = \{x: d(x,A) \geq h\}$.
They use variable $X_a = f(a) = d(a,A)$ for subspaces $A$ such that $\mathbb{P}(x \in A) \geq 1/2$, which implies that $m(X) = 0$.
With these choices, (\ref{skewed tail distribution}) implies that $\mathbb{P}(x \in A^{-h}) \leq e^{-(h/\sigma - 1)^2/2}$ when $h \geq \sigma$.
If we choose $f(x) = \min\{h, d(x,A)\}$, then the same choices imply $\mathbb{E}(X) \leq h/2$, and so (\ref{tail distribution}) gives a better bound when $h < 2\sigma$: $\mathbb{P}(x \in A^{-h}) \leq e^{-(h/\sigma)^2/8}$.

The \emph{subgaussian constant} for $X$ is a value for $\sigma_X^2$ such that (\ref{subgaussian ineq}) is true for all real $t$.
For a probability space $S$, $\sigma_S^2$ is the supremum of $\sigma_f^2$, where $f$ is over all Lipschitz functions over $S$.
When it is clear, we drop the subscript from $\sigma$.
With all of these amazing inequalities, all we have left to do is find distributions $X$ and spaces $S$ with reasonable values for $\sigma$.

Suppose $C$ is any finite set, and our space is $C^n$ equipped with the Hamming distance ($d_H(x,y) = |\{i: x_i \neq y_i\}|$).
If $X_a=f(a)$ for Lipschitz function $f$ over $x \in C^n$, then McDiarmid's inequality implies $\sigma^2 \leq n/4$ in (\ref{tail distribution}).
Bobkov, Houdr\'e, and Tetali \cite{BHT} showed that if $|C|$ is even (including the already-known case of $|C| = 2$), then $\sigma^2 = n/4$.
If $|C|$ is odd, then $n\sigma^{-2} = 2|C|\log\left(\frac{|C|+1}{|C|-1}\right)$.
The symmetric group on $n$ elements is a subset of $C^{n}$ for $|C| = n$.
Bobkov, Houdr\'e, and Tetali \cite{BHT} showed that if our space is the symmetric group with a distance inherited from the Hamming distance on $C^n$ in this manner, then $\sigma^2 \leq n-1$.

We are also interested in probability spaces $X = f \circ \mu$ defined as uniform probabilities $\mu$ over the vertex set of a graph $G$ equipped with functions $f:V(G) \rightarrow \mathbb{R}$ that are Lipschitz with respect to the standard graph distance in $G$.
Alon, Boppana, and Spencer \cite{ABS} demonstrated that the subgaussian constant \emph{tensors out} with respect to the Cartesian product.
As this result will be important to us (and because the proof is so short and elegant), we include it here for completeness.

First, let us establish notation that we will use again later.
For a graph $G$, let $\Omega_G$ be the set of Lipschitz functions over $G$.
For $f \in \Omega_G$, the \emph{log-moment function} of $f$ is $L_f(t) = \ln(\mathbb{E}(e^{t(f - \mathbb{E}(f))}))$.
The log-moment function of $G$ is  $L_G(t) = \max_{f \in \Omega}L_f(t)$.
(Let us quickly explain why we use $\max$ instead of $\sup$: the log-moment function is invariant under translation---$L_f(t) = L_{f+c}(t)$---and therefore we may restrict ourselves to functions where $f(v) = 0$ for arbitrary fixed vertex $v$.  This restricted subset of $\Omega$ is a closed bounded subset of $\mathbb{R}^n$ and therefore compact.)
We can thus calculate the subgaussian constant as $\sigma_G^2 = \sup_{t>0} 2L_G(t)t^{-2}$.

\begin{theorem}[\cite{ABS}]\label{subgaussian tensors out}
For graphs $G$ and $H$, we have that $\sigma_{G \square H}^2 = \sigma_G^2 + \sigma_H^2$.
\end{theorem}
\begin{proof}
Fix a $t$.

Let $f,f'$ be Lipschitz functions over $G,H$ such that $L_G(t) = L_f(t), L_H(t) = L_{f'}(t)$ respectively.
Now define $\ell$ to be the function $\ell(u,v) = f(u) + f'(v)$ over $G \square H$, which is Lipschitz by choice of $f,f'$.
By translation invariance, let us assume that $\mathbb{E}(f) = \mathbb{E}(f') = 0$, which implies that $\mathbb{E}(\ell) = 0$ as well.
Notice that 
$$ L_{G \square H}(t) \geq L_\ell(t) = \ln(\mathbb{E}_{u,v}(e^{t(f(u) + f'(v))})) = \ln\left(\mathbb{E}_{u}(e^{tf(u)})\mathbb{E}_{v}(e^{tf'(v)})\right) = L_f(t) + L_{f'}(t).$$
As this holds for each $t$, we conclude that $\sigma_{G \square H}^2 \geq \sigma_G^2 + \sigma_H^2$.

Now suppose that $\ell$ is a Lipschitz function over $G \square H$ such that $L_{G \square H}(t) = L_\ell(t)$.
Let $f(u) = \mathbb{E}_v (\ell(u,v))$, and define a family of functions over $H$ as $f'_u$ such that $\ell(u,v) = f(u) + f'_u(v)$.
Because $\ell$ is Lipschitz and expectation is linear, we have that $f$ is Lipschitz.
For a fixed $u$, we have that $f(u)$ is constant, and therefore $f_u'$ is also Lipschitz.
The expectation of $f$ over $u$ is the expectation of $\ell$ over $(u,v)$.
For simplicity, we will use translation invariance of the log-moment function to assume $\ell$ has expectation $0$, and therefore $f$ has expectation $0$ as well.
By definition of $f$, it follows that $f'_u$ has expectation $0$ for each $u \in V(G)$.
The theorem then follows from the monotonicity of the exponential function and that for all $t$, 
\begin{eqnarray*}
e^{L_{G \square H}(t)}  & = &  \mathbb{E}_{u,v}( e^{t\ell(u,v)} ) \\
			& = &  \mathbb{E}_u\left( e^{tf(u)} \mathbb{E}_v(e^{tf_u(v)} | u) \right) \\
			& \leq & e^{L_G(t)} e^{L_H(t)}.
\end{eqnarray*}
\end{proof}

\subsection{Concentration of permutations and levels of the Boolean lattice}
We are interested in Ollivier and Villani's \cite{OV} use of concentration inequalities.
They use a result by Bobkov, Houdr\'e, and Tetali \cite{BHT} on the concentration of $S_n$ as a lemma, and hence we wish to establish this result with great care.
The proof of this statement is omitted from the published paper, but its proof is available in an extended manuscript available on a personal website.
We include the proof here for self-containment and to establish the context necessary for Remark \ref{strict inequality}.

\textbf{Theorem \ref{subgaussian permutations} } (\cite{BHT}) \textit{
Let $S_n$ be the symmetric group on $n$ elements equipped with the Hamming distance $d_H$.
The subgaussian constant for this space satisfies $\sigma_{S_n}^2 \leq n-1$.
}
\begin{proof}
We prove this by induction on $n$.
The base case follows from $|S_1| = 1$.
Also, $S_2 \cong K_2$ except with distances doubled, and therefore $\sigma_{S_2}^2 = 4 \sigma_{K_2}^2 = 1$.
Next, we will show that $\sigma^2 _{S_n} \leq \sigma^2 _{S_{n-1}} + 1$.

Let $T_i = \{\sigma \in S_n : \sigma(1) = i\}$, so that $T_1, \ldots , T_n$ partition $S_n$ and each $T_i$ is isomorphic to $S_{n-1}$.
Recall the log-moment function $L_{S_n}(t)$ and let $f$ be a function that is Lipschitz over $S_n$ and $L_f(t) = L_{S_n}(t)$.
Let $f_i$ denote $f$ restricted to the domain of $T_i$.
We introduce the translated log-moment function $L_f^*(t) = \ln(\mathbb{E}(e^{tf}))$, so that $e^{L_f(t)} = e^{L_f^*(t)}e^{-t\mathbb{E}f}$.
It follows that 
\begin{eqnarray*}
e^{L_f(t)}	& = & e^{-t\mathbb{E}f} e^{L_f^*(t)}\\
		& = & e^{-t\mathbb{E}f} \mathbb{E}e^{tf} \\
		& = & e^{-t\mathbb{E}f} \frac{1}{n} \sum_{i=1}^n \mathbb{E}e^{tf_i} \\
		& \leq & e^{-t\mathbb{E}f} e^{(n-2)t^2/2} \frac{1}{n} \sum_{i=1}^n e^{t\mathbb{E}f_i}, \\
		& = & e^{(n-2)t^2/2} \frac{1}{n} \sum_{i=1}^n e^{t(\mathbb{E}f_i - \mathbb{E}f)}
\end{eqnarray*}
where the second-to-last part follows by induction.
We will show that $\frac{1}{n} \sum_{i=1}^n \mathbb{E}e^{t(\mathbb{E}f_i - \mathbb{E}f)} \leq e^{t^2/2}$, which will prove the theorem.

Let $g$ be a variable such that $g(i) = \mathbb{E}f_i$.
Note that $\mathbb{E}(g) = \mathbb{E}(f)$, and therefore $e^{L_g(t)} = \frac{1}{n} \sum_{i=1}^n e^{t(\mathbb{E}f_i - \mathbb{E}f)}$.
So we are trying to prove that $\sigma_g^2 \leq 1$.
We will show that for all $i,j$ we have that $|g(i) - g(j)| \leq 2$, and therefore $\sigma_g^2 \leq 4 \sigma_{K_n}^2 \leq 1$, which will prove the theorem.

To show that $|\mathbb{E}_{x \in T_i}(f(x)) - \mathbb{E}_{x \in T_j}(f(x))| \leq 2$ for $i \neq j$ we establish a bijection $R_{i,j}:T_i \rightarrow T_j$ such that $d_H(x,R_{i,j}(x)) \leq 2$ for all $x \in T_i$.
For a fixed $x \in T_i$, let $k = x^{-1}(j)$.
Let $y = R_{i,j}(x)$, and define $y(1) = j$, $y(k) = i$, and $y(\ell) = x(\ell)$ otherwise.
\end{proof}

\begin{remark}\label{strict inequality}
Theorem \ref{subgaussian permutations} is not sharp because $4 \sigma_{K_n}^2 < 1$ whenever $n$ is odd.
Because the proof is inductive, the bound on $\sigma_{S_n}^2$ could be improved by the sum of $1 - 4 \sigma_{K_m}^2$ over the odd $m \leq n$.
Unfortunately, $1 - 4 \sigma_{K_m}^2 \leq O(m^{-2})$, and thus this improvement is finite.
\end{remark}

Let $J_n = K_n \square K_{n-1} \square \cdots \square K_2$.
Recall that Bobkov, Houdr\'e, and Tetali \cite{BHT} gave the exact subgaussian constant for the complete graph, and thus Theorem \ref{subgaussian tensors out} implies that 
$$\sigma_{J_n}^2 = \left(n - 1 - \sum_{3 \leq 2r+1\leq n} 1 - \frac{2}{(2r+1)\log\left(\frac{r+1}{r}\right)} \right)/4 < \frac{n-1}4. $$
The extended manuscript of Bobkov, Houdr\'e, and Tetali \cite{BHT} includes a proof from Schechtman that $\sigma_{S_n}^2 \leq 9 \sigma_{J_n}^2$ by constructing a $3$-Lipschitz function $T:J_n \rightarrow S_n$.
It can be observed that $T^{-1}$ is $2$-Lipschitz, and therefore $\sigma_{S_n}^2 \geq \sigma_{J_n}^2/4$.

We know the extremal function for the log-moment function of $J_n$.
By the proof of Theorem \ref{subgaussian tensors out} the extremal function of $J_n$ is the extremal function of $K_i$ tensored out for $2 \leq i \leq n$, and Bobkov, Houdr\'e, and Tetali \cite{BHT} gave the extremal function for $K_i$.
Up to symmetry, the extremal function for $J_n$ is then $f(x_1, \ldots, x_{n-1}) = |\{i:2x_i>n+1-i\}|$.

We can construct close lower bounds on $\sigma_{S_n}$ by using similar functions.
That is, we consider functions $f$ that are sum of indicator variables with balanced probabilities.
Bobkov, Houdr\'e, and Tetali \cite{BHT} show that $\sigma_{S_n}^2 > (n-1)/16$ using the function $f(\pi) = |\{i : i \leq \lfloor n/2 \rfloor, \pi(i) > \lfloor n/2 \rfloor \}|$, which can easily be expressed as the sum of $\lfloor n/2 \rfloor$ indicator variables.
Recall for $X = \sum_{i=1}^n X_i$ that $\var(X) = \sum_{i=1}^n \var(X_i) + \sum_{i \neq j} \cov(X_i,X_j)$.
We improve the previous lower bound on $\sigma_{S_n}$ by finding indicator variables with positive covariance.

\begin{theorem} \label{at least n/4 for S_n}
Let $S_n$ be the symmetric group on $n \geq 3$ elements equipped with the Hamming distance $d_H$.
The subgaussian constant for this space satisfies $\sigma_{S_n}^2 > \frac{n}{4} > \sigma_{J_n}^2 + \frac{1}{4}$.
\end{theorem}
\begin{proof}
For a permutation $\pi \in S_n$, let $X_i$ be the indicator variable that $\pi(i) > n/2$.
Bobkov, Houdr\'e, and Tetali's function is $f(\pi) = \sum_{i=1}^{\lfloor n/2 \rfloor} X_i(\pi)$.
They demonstrated that $\var(f) > (n-1)/16$, and so their result follows from Theorem \ref{var less than subgauss}.
Moreover, if $n$ is even, then $\var(f) > n/16$.

Consider the function $X = \sum_{i=1}^{\lfloor n/2 \rfloor} X_i + \sum_{\lfloor n/2 \rfloor + 1 }^{n} (1 -  X_i)$.
If $n$ is even, then because $\pi$ is a permutation, $\sum_{i=1}^{ n/2 } X_i = \sum_{ n/2  + 1}^{n} (1 -  X_i)$.
So when $n$ is even, $X(\pi) = 2f(\pi)$, and so $\var(X) = 4 \var(f) > n/4$.

Now suppose $n = 2k+1$ for integer $k \geq 1$.
A direct calculation gives us that $\mathbb{E}(X_i) = \frac{k+1}{2k+1}$ and $\var(X_i) = \var(1-X_i) = \frac{k(k+1)}{(2k+1)^2}$.
Moreover, for $i \neq j$, $\cov(1-X_i, 1-X_j) = \cov(X_i, X_j) = - \frac{k+1}{2(2k+1)^2}$ and $\cov(X_i,1-X_j) = - \cov(X_i, X_j)$.
Therefore 
\begin{eqnarray*}
 \var(X) &=& (2k+1) \frac{k(k+1)}{(2k+1)^2} + \left(k(k+1) - {k \choose 2} - {k+1 \choose 2}\right) \frac{k+1}{2(2k+1)^2}\\
	& = & \frac{n}{4} + \frac{(n-1)^2-2}{8n^2}. 
\end{eqnarray*}

\end{proof}

Let $C_r = \{A \in \{0,1\}^n : |A| = r\}$ and $C_{r_1, r_2, \ldots, r_k} = \cup_{i} C_{r_i}$.
Ollivier and Villani \cite{OV} established that $\sigma_{C_{n/2}}^2 \leq n-1$ and $\sigma_{C_{(n-1)/2, (n+1)/2}}^2 \leq n$, which we wish to generalize.

\begin{theorem}\label{different level sets}
For $n \geq 3$, 
$\sigma_{C_{(n-r)/2,(n+r)/2}}^2 < n - 1 + r^2/4$
\end{theorem}
\begin{proof}
Let $f$ be a Lipschitz function on $C_{(n-r)/2,(n+r)/2}$.
Let $H = rK_2 \square S_n$, where $rK_2$ is two points $\{0,r\}$ distance $r$ apart.
We define a surjection $\psi:H \rightarrow C_{(n-r)/2,(n+r)/2}$ as follows: $\psi(0,\pi) = f(\pi(\{1,2,\ldots, (n-r)/2\}))$ and $\psi(r,\pi) = f(\pi(\{1,2,\ldots, (n+r)/2\}))$.
Let $g: H \rightarrow \mathbb{R}$ such that $g(h) = f(\psi(h))$.
Because ${n \choose (n-r)/2} = {n \choose (n+r)/2}$, the pre-image under $\psi$ of any point in $C_{(n-r)/2,(n+r)/2}$ is of fixed size.
Therefore $\mathbb{E}(e^{t(g - \mathbb{E}g)}) = \mathbb{E}(e^{t(f - \mathbb{E}f)})$.

We claim the following two facts: (A) $g$ is Lipschitz over $H$ and (B) $\sigma_H^2 < n-1 + r^2/4$.
The claim implies that $\mathbb{E}(e^{t(g - \mathbb{E}g)}) < e^{t^2( n-1 + r^2/4)/2}$, which proves the theorem.

By the definition of Cartesian product, $g$ is Lipschitz if and only if it is Lipschitz in each coordinate.
Because $f$ is Lipschitz and we apply the Hamming distance to $S_n$, we have that $|g(x,\pi) - g(x,\pi')| \leq d_{S_n}(\pi, \pi')$.
Because $\psi(0,\pi) \subseteq \psi(r,\pi)$, $|\psi(r,\pi) - \psi(0,\pi)| = r$, and $f$ is Lipschitz, we have that $|g(x,\pi) - g(x',\pi)| \leq r$.
Thus, $g$ is Lipschitz.
This proves (A).

By Theorem \ref{subgaussian tensors out}, $\sigma_H^2 = \sigma_{S_n}^2 + \sigma_{rK_2}^2$.
By Remark \ref{strict inequality}, $\sigma_{S_n}^2 < n-1$.
By the definition of subgaussian constant, $\sigma_{rK_2}^2 = r^2\sigma_{K_2}^2$.
Recall that $\sigma_{K_2}^2 = 1/4$.
This proves (B).
\end{proof}

Theorem \ref{different level sets} is strongest when $r \leq O(\sqrt{n})$.
We will need another result for when $r = \epsilon n$ for small $\epsilon > 0$.
We restrict our attention to the midpoints of distant sets of vertices, and only establish an extension of (\ref{tail distribution}).

In the following we will need to bound from below the value $|C_{R}| = {n \choose R}$ for $R = n/2 - r$ and $\sqrt{n} \ll r$.
The standard inequality to use is $(\frac{n}{R})^R \leq {n \choose R}$, which gives $|C_R| \geq (2+O(\epsilon))^{n/2-r}$ when $r = \epsilon n$.
We need (and produce) the stronger inequality $|C_R| \geq 2^{n-O(r)}$.

\begin{theorem} \label{Linearly far from center}
Suppose $c \geq 2$, $n \geq 3$, $\frac{c-1}{2(c+1)}n > R > \sqrt{n \ln(\frac{c}{c-1})}$, and $|n/2 - r_i| \leq R$ for all $i$.
Let $X_*$ be a Lipschitz function over $\{0,1\}^n$, and let $X$ be $X_*$ induced on $C_{r_1, \ldots, r_k}$.
$$ \mathbb{P}(X - \mathbb{E}X_* \geq h \geq 0) < e^{(R+3)\ln(c) - \ln(k) -\frac{h^2}{2(n-1)}}. $$
In particular, if $n/6 > R > \sqrt{n \ln(2)}$, then 
$$ \mathbb{P}(X - \mathbb{E}X_* \geq h \geq 0) < e^{(R+3)\ln(2) - \ln(k) -\frac{h^2}{2(n-1)}}. $$
\end{theorem}
\begin{proof}
Because $C_{r_1, \ldots, r_k} \subset \{0,1\}^n$, every event $X - \mathbb{E}X_* \geq h$ is also an event $X_* - \mathbb{E}X_* \geq h$.
If we wish to count such events, this inequality becomes 
$$\mathbb{P}(X - \mathbb{E}X_* \geq h \geq 0)|C_{r_1, \ldots, r_k}| \leq \mathbb{P}(X_* - \mathbb{E}X_* \geq h \geq 0)|\{0,1\}^n|.$$
By Remark \ref{strict inequality}, inequality (\ref{tail distribution}) applied to $X_*$ is 
$$ \mathbb{P}(X_* - \mathbb{E}X_* \geq h \geq 0) < e^{ -\frac{h^2}{2(n-1)}}. $$
Let $C = C_{n/2-R, \ldots, n/2+R}$.
A standard inequality tells us that $|C|2^{-n} \geq 1 - e^{-R^2/n}$, and by the assumption $R > \sqrt{n \ln(\frac{c}{c-1})}$ we have that $|C| > 2^{n}/c$.
We will show that $|C_{n/2-R}| \geq c^{-R-2}|C|$, which will imply that $|C_{r_1, \ldots, r_k}| \geq k c^{-R-2} |C|$, which implies the theorem.

Because $\frac{c-1}{2(c+1)}n > R$ and $i \geq 0$ we have that $\frac{|C_{n/2-R+i+1}|}{|C_{n/2-R+i}|} = \frac{n/2+R-i}{n/2-R+i+1} < c$.
Therefore $|C_{n/2-R+i}| < c^i |C_{n/2-R}|$, and because $c \geq 2$ this implies
$$|C| < 2|C_{n/2-R}|\sum_{i=0}^{R} c^i < |C_{n/2-R}|c^{R + 2}.$$
\end{proof}

In the end, our concentration of measure will be used in the following manner.

\begin{theorem} \label{concentration on levels of Boolean lattice}
Let $A,B \subseteq \sigma_{C_{(k-r)/2,(k+r)/2}}^2 \subset \{0,1\}^k$ such that $|A| \leq |B|$ and $\min_{a \in A,b \in B}d(a,b) \geq t$.
Under these circumstances, 
$$|A| \leq |C_{(k-r)/2,(k+r)/2}| e^{\frac{-t^2}{8k - 8 + 2r^2}}.$$
If $r = \epsilon k$ and $t = \delta k$ with fixed coefficients that satisfy $C = \delta^2 - 8\ln(2)\epsilon > 0$, then 
$$|A| \leq |C_{(k-r)/2,(k+r)/2}| e^{-Ck(1 + o(1))}.$$
\end{theorem}
\begin{proof}
Let $X_*$ be a variable over $\{0,1\}^k$ such that $X_*(u) = d(u,A)$.
By this construction, $X_*$ is Lipschitz, $A = \{u: X_*(u) \leq 0\}$, and $B \subseteq \{u: X_*(u) \geq t\}$.
Let $X$ be $X_*$ induced on $C_{(k-r)/2,(k+r)/2}$.
If $\mathbb{E}(X) \leq t/2$, then apply Theorem \ref{different level sets} to inequality (\ref{tail distribution}) on variable $X$  to see that $|B| \leq |C_{(k-r)/2,(k+r)/2}| e^{-t^2/(8(k - 1 + r^2/4))}$.
If $\mathbb{E}(X) \geq t/2$, then apply Theorem \ref{different level sets} to inequality (\ref{tail distribution}) on variable $-X$ to see that $|A| \leq |C_{(k-r)/2,(k+r)/2}| e^{-t^2/(8(k - 1 + r^2/4))}$.
Because $|B| \geq |A|$ by assumption, the first part of the theorem follows.

The diameter of $\{0,1\}^k$ is $k$, and so we may assume that $\delta \leq 1$.
If $\delta^2 - 8\ln(2)\epsilon > 0$, then $\epsilon < 1/6$.
The second part of the theorem follows similarly, with Theorem \ref{Linearly far from center} (specifically, the second inequality with $c=k=2$) replacing Theorem \ref{different level sets} and inequality (\ref{tail distribution}).

\end{proof}

\subsection{Concentration of cycles}
A variable $X$ is called \emph{optimal} if $L_G(t) = L_X(t)$.
Let $\Omega^*$ be the variables in $\Omega$ that are not a convex combination of other variables in $\Omega$. 
The log-moment function is a summation of functions that are strictly convex in each variable composed with a monotone function, and therefore the optimal functions are contained by $\Omega^*$.
Bobkov, Houdr\'e, and Tetali \cite{BHT} established the following combinatorial fact.
If $E_X = \{uv: uv \in E(G), |X(u) - X(v)| = 1\}$, then
\begin{equation} \label{spanning tree}
\mbox{if $X \in \Omega^*$, then $E_X$ contains a spanning tree.}
\end{equation}
The log-moment function is also symmetric, and Bobkov, Houdr\'e, and Tetali used this to show a second combinatorial fact.
If $\pi$ is a permutation of $V(G)$ and $X \circ \pi \in \Omega$, then 
\begin{equation} \label{permuted values}
X \mbox{ is optimal if and only if } X' \mbox{ is optimal.}
\end{equation}

Bobkov, Houdr\'e, and Tetali \cite{BHT} used (\ref{spanning tree}) and (\ref{permuted values}) to quickly calculate the subgaussian constant of $C_n$ when $n$ is even.
However, the proof is just short of calculating $\sigma_G$ for $G = C_{2k+1}$.
Up to symmetry, the cycle $C_n$ contains a unique spanning tree, which contains all but one edge (by symmetry, call this edge $uv$).
Let $X \in \Omega^*$, and see that $|X(u) - X(v)| \cong n+1 (\modd 2)$.
Because $X$ is Lipschitz and $uv \in E(C_n)$, we also know that $|X(u) - X(v)| \leq 1$.
If $n=2k$, then $uv \in E_X$, and therefore $E_X = E(C_{2k})$.
Moreover, if there exists three distinct vertices $i, j, k$ such that $X(i) = X(j) = X(k)$, then there exists a permutation $\pi$ such that $X \circ \pi \in \Omega$ and $E_{X \circ \pi}$ does not contain a spanning tree (thus violating (\ref{permuted values})).
Up to symmetry, there exists a unique integer-valued Lipschitz function on $C_{2k}$ such that no value appears in the image more than twice and the endpoints of each edge differ by exactly $1$; and therefore $X$ is known, and therefore $\sigma_{2k}$ is known.

Conjecture \ref{spread of cycles}, made by Bobkov, Houdr\'e, and Tetali \cite{BHT} and repeated by Sammer and Tetali \cite{ST}, is that a similar construct is optimal for odd cycles.
In order to enhance this argument for odd cycles, we will argue a stronger statement than (\ref{spanning tree}) and (\ref{permuted values}).
We will also use this stronger statement again at a later point in the paper.
Let $S_{V(G)}$ denote the set of permutations on $V(G)$ and 
$$ \Omega^\circ = \Omega^* - \{\lambda Z \circ \pi_Z + (1-\lambda) Y \circ \pi_Y: \lambda \in (0,1), Z,Y \in \Omega, Z \neq Y, \pi_Z, \pi_Y \in S_{V(G)} \}.$$

To explain $\Omega^\circ$ in simpler terms, we have that $\Omega^*$ is $\Omega$ minus functions that are convex combinations of other Lipschitz functions, and $\Omega^\circ$ is $\Omega$ minus functions that are convex combinations of permutations of other Lipschitz functions (which might not be Lipschitz after the permutation is applied).
We do not know apriori that $\Omega^\circ$ is non-empty, but this follows from the following Lemma.

\begin{lemma}\label{not convex combo}
If $X$ is optimal, then $X \in \Omega^\circ$.
\end{lemma}

\begin{proof}
The log-moment function is well defined for all functions $X:V(G) \rightarrow \mathbb{R}$, not just those inside $\Omega$.
Moreover, the log-moment function is strictly convex over this extended domain.
So if $X$ is a convex combination of $Z \circ \pi_Z$ and $Y \circ \pi_Y$, then $L_X(t) < \max\{ L_{Z \circ \pi_Z}(t), L_{Y \circ \pi_Y}(t)\}$.
And by symmetry,  $L_{Z \circ \pi_Z}(t) =  L_{Z}(t)$ and $L_{Y \circ \pi_Y}(t) =  L_{Y}(t)$.
\end{proof}

\begin{theorem} \label{odd cycle optimal}
Conjecture \ref{spread of cycles} is true.
\end{theorem}
\begin{proof}
Let $V(C_n) = \{u_1, \ldots, u_n\}$, and let the indices be taken modulo $n$.
Fix some $i$, and let $X_i(w) = d(u_i, w)$.
We will show that $\Omega^\circ$ is the set of translations and reflections of $X_i$ for $1 \leq i \leq n$. 
Suppose $X \in \Omega^{\circ}$.

One method to characterize the translations and reflections of $X_i$ is as the family of Lipschitz functions $\{Y_k\}$ that satisfy the condition that for any $\ell$, 
\begin{equation} \label{up and then down}
  |\{j :  Y_k(u_j) = \ell\}| \leq 2.
\end{equation}
By translation invariance, let us assume that $X$ is integer valued.
We will show that $X$ satisfies (\ref{up and then down}).

Let $m = \min_i X(i)$ and $M = \max_i X(i)$.
There exists an $X' = X \circ \pi$ for a permutation $\pi$ of $V(C_n)$ such that \\
(1) for $1 \leq i \leq 1 + M - m$ we have that $X'(u_i) = m + i - 1$, and \\
(2) for $1 + M - m \leq j \leq n$ we have that $X'(u_j) \geq X'(u_{j+1})$.\\
By construction, $X' \in \Omega$.
By (\ref{permuted values}), $X'$ is optimal, and so by (\ref{spanning tree}), $X'(u_i) = X'(u_{i+1}) + 1$ for $i \in [1 + M - m, n] - \{\ell'\}$ for some value $\ell'$.

If $X'(u_{\ell'}) = X'(u_{\ell'+1}) + 1$ or $\ell' = 1 + M - m$ or $\ell' = n$, then we are done.
So assume $X'(u_{\ell'}) = X'(u_{\ell'+1})$ and $2 + M - m \leq \ell' \leq n-1$.
Consider the functions $X_*$ and $X_{**}$ where\\
(A) $X_*(u_i) = X'(u_i)$ when $i \neq \ell'$ and $X_*(u_{\ell'}) = X'(u_{\ell'})+1$, and\\
(B) $X_{**}(u_i) = X'(u_i)$ when $i \neq \ell' + 1$ and $X_{**}(u_{\ell'+1}) = X'(u_{\ell'+1})-1$.\\
By construction, $X_*, X_{**} \in \Omega$.
Moreover, if $\pi$ is the permutation on $V(C_n)$ that transposes $\ell'$ with $\ell'$, then $X' = \frac12(X_* \circ \pi + X_{**})$.
By Lemma \ref{not convex combo}, this implies that $X'$ is not optimal.
\end{proof}

\subsection{The structure of the subgaussian constant and spread}

Let us use Lemma \ref{not convex combo} to prove a statement that will be used later.
A \emph{hair} of $G$ is a sequence of vertices $w_0, w_1, \ldots, w_k$ such that $N(w_i) = \{w_{i-1}, w_{i+1}\}$ for $1 \leq i \leq k-1$ and $N(w_k) = \{w_{i-1}\}$.
We may use (\ref{spanning tree}) and (\ref{permuted values}) to state that 
\begin{equation}\label{unimodular hairs}
\mbox{any optimal function $X$ is unimodular on any hair.}
\end{equation}
That is, for hair $w_0, \ldots, w_k$ and optimal function $X$, there exists an $0 \leq \ell \leq k$ such that \\
(A) $X(w_{i})-X(w_{i-1}) = X(w_{i'}) - X(w_{i'-1})$ for all $1 \leq i,i' \leq \ell$, and\\
(B) $X(w_{j})-X(w_{j-1}) = X(w_{j'}) - X(w_{j'-1})$ for all $\ell < j,j' \leq k$.\\
We will need to find the set of optimum functions of a particular family of trees in a later section; for this purpose we will need to know that ``small'' hairs are monotone.

\begin{lemma}\label{hairs go one way}
Let $X$ be an optimal Lipschitz variable on $G$, and let $w_0, w_1, \ldots, w_k$ be a hair of $G$.
Let $m = \min_i X(w_i)$ and $M = \max_i X(w_i)$.
If $G$ has vertices $u_1, \ldots u_{M-m-1}$ such that $X(u_i) = m+i$ and $\{u_1, \ldots, u_{M-m-1}\} \cap \{w_1, \ldots, w_k\} = \emptyset$, then $X(w_{i})-X(w_{i-1}) = X(w_{i'}) - X(w_{i'-1})$ for all $1 \leq i,i' \leq k$.
\end{lemma}
\begin{proof}
By translation invariance, let us assume that $X$ takes integral values.
By the Lipschitz condition, for any path $v_1 v_2 \ldots v_k$ satisfies $\mathbb{Z} \cap [X(v_1), X(v_k)] \subseteq \{X(v_1), X(v_2), \ldots, X(v_k)\}$.
We directly conclude that the sequence $X(w_0), X(w_1), \ldots, X(w_k)$ must contain each integer in the range $[m, M]$.
But by partitioning the hair into multiple paths, we also see that each integer in $(m, X(w_0)]$ or in $[X(w_0),M)$ appears at least twice.

By symmetry, let us assume that each integer in $(m, X(w_0)]$ appears at least twice.
There exists a permutation $\pi$ such that \\
(A) $X(w_{\pi(i)}) = X(w_0) - i$ for $0 \leq i \leq X(w_0) - m$, and \\
(B) $X(w_{\pi(i)}) \leq X(w_{\pi(j)})$ when $X(w_0) - m \leq i \leq j$.\\
Let $X' = X \circ \pi'$, where $\pi'$ is a permutation on $V(G)$ that fixes vertices outside of the hair and permutes vertices inside the hair according to $\pi^{-1}$.
By construction, $X'$ is Lipschitz, and so  (\ref{permuted values}) implies that $X'$ is optimal.
Any spanning tree of $G$ must contain each edge between consecutive vertices of a hair, and (\ref{spanning tree}) implies that $X(w_{\pi(X(w_0) - m + j)}) = m + j$.

So each value in the range $(m, X(w_0)]$ appears exactly twice, each value in the range $(X(w_0), M]$ appears exactly once, and the value $m$ appears exactly once.

Now let us assume that $X$ is unimodular but not monotone.
By considering a sub-hair and applying symmetry, let us assume that for $i \geq 1$ we have that $X(w_i) = X(w_0) + i - 2 = m + i - 1$.
Let $\ell = M - m + 1$ so that the hair is vertices $w_0, w_1, \ldots, w_\ell$.
We will prove that if there exist vertices $u_1, \ldots, u_{\ell-2}$ such that $X(u_i) = m + i$, then there exists Lipschitz functions $X_+$ and $X_-$ and permutations $\pi_+$ and $\pi_-$ such that $X = \frac{1}{2}(X_+ \circ \pi_+ + X_- \circ \pi_-)$.
By Lemma \ref{not convex combo}, this will prove the second part of the theorem.

Let $X_+(v) = X(v)$ for vertices $v$ not in the hair, $X_+(w_0) = X(w_0)$, $X_+(w_\ell) = X(w_\ell) = M$, and $X_+(w_i) = X(w_i) + 2$ for $1 \leq i < \ell$.
Let $X_-(v) = X(v)$ for vertices $v$ not in the hair, $X_+(w_0) = X(w_0)$, $X_+(w_\ell) = X(w_\ell) = M$, and $X_+(w_i) = X(w_i) + 2$ for $1 \leq i < \ell$.
Now we will define bijections $\pi_- ^{-1}$ and $\pi_+^{-1}$ on $\{w_1, \ldots, w_\ell, u_1, \ldots, u_{\ell-2}\}$ such that $X(u)-1 = X_-(\pi_- ^{-1}(u)$ and $X(u)+1 = X_+(\pi_+ ^{-1}(u)$.
This will suffice, as $X, X_+, X_-$ are equal on all other vertices.
We define 
\begin{itemize}
	\item for $1 \leq i \leq \ell - 2$, set $\pi_-^{-1}(w_i) = w_{i+1}$ and $\pi_+^{-1}(w_i) = u_i$, 
	\item set $\pi_-^{-1}(u_1) = w_{1}$ and $\pi_+^{-1}(u_1) = w_1$,
	\item for $2 \leq i \leq \ell - 2$, set $\pi_-^{-1}(u_i) = u_{i-1}$ and $\pi_+^{-1}(w_i) = w_i$, 
	\item set $\pi_-^{-1}(w_{\ell-1}) = w_{\ell}$ and $\pi_+^{-1}(w_{\ell-1}) = w_{\ell}$, and
	\item set $\pi_-^{-1}(w_{\ell}) = u_{\ell-2}$ and $\pi_+^{-1}(w_{\ell}) = w_{\ell-1}$.
\end{itemize}

\end{proof}

\begin{remark}\label{hair matched by path}
The assumption in the second part of Lemma \ref{hairs go one way} can be relaxed to only assuming that $u_1$ and $u_{\ell-2}$ exist and satisfy $X(u_1) \leq m+1$ and $X(u_{\ell-2}) \geq M-1$.
This is because the rest of the vertices can be found on a shortest path from $u_1$ to $u_{\ell-2}$.
\end{remark}

We will call a function $X$ \emph{variance-optimal} if $\var(X) = c^2(G)$.

\begin{remark}\label{variance too}
Let us first note that the variance function is strictly convex and symmetric.
Therefore (\ref{spanning tree}), (\ref{permuted values}), Lemma \ref{not convex combo}, (\ref{unimodular hairs}), and Lemma \ref{hairs go one way} all hold when optimal is replaced with variance-optimal.
The analogue of Theorem \ref{subgaussian tensors out} also follows from minor modifications to the proof.
\end{remark}

Variance also satisfies a handful of other properties.
For example, variance is even, so $X$ is variance-optimal if and only if $-X$ is variance-optimal.
Next, we present a slightly stronger version of Theorem \ref{defining the origin}, where the improvement will be crucial to establishing Theorem \ref{variance optimal is not isoperimetric}.

\begin{theorem}\label{go away from average}
If $\var(X) = c^2(G)$ and $X(u) \geq \mathbb{E}(X)$, then there exists $v \in N(u)$ such that $X(u) = X(v) + 1$.
If $X(u) \geq \mathbb{E}(X) - 0.5$ and $|V(G)| \geq 10$, then there exists $v \in N(u)$ such that $X(u) \geq X(v)$.
\end{theorem}
\begin{proof}
For $0 < \epsilon$, define $X_\epsilon$ to be the variable $X_\epsilon(w) = X(w)$ when $w \neq u$ and $X_\epsilon(u) = X(u) + \epsilon$.
To prove the theorem, we will show that when the assumptions are violated there exists a value of $\epsilon$ such that $\var(X_\epsilon) > \var(X)$ and that $X_\epsilon$ is Lipschitz.
Recall from Theorem \ref{defining the origin} that $X(v) - X(u)$ is an integer.
If $\epsilon < 1$ and for all $v \in N(u)$ we have that $X(v) \geq X(u)$, then $X_\epsilon$ is Lipschitz.
If $\epsilon < 2$ and for all $v \in N(u)$ we have that $X(v) \geq X(u) + 1$, then $X_\epsilon$ is Lipschitz.

A direct calculation gives us that $\frac{d}{d \epsilon} \var(X_\epsilon) = \epsilon\left(1 + 2 p(1-p)) + 2p(X(u) - \mathbb{E}(X))\right)$, where $p = 1/|V(G)|$ is the probability of $u$.
Thus $\frac{d}{d \epsilon} \var(X_\epsilon) > 0$ when $X(u) > \mathbb{E}(X)$, and the first part of the theorem follows.
If $|V(G)| \geq 10$, then $\frac{d}{d \epsilon} \var(X_\epsilon) \geq \epsilon - 0.1$, and so integrating from $\epsilon = 0$ to $\epsilon = 2$ gives a positive total change.
Thus the second part of the theorem holds.
\end{proof}

One interpretation of Theorem \ref{defining the origin} is that the intuition of Conjecture \ref{spread of cycles} is true for variance-optimal functions \emph{for all graphs}.
The intuition of extremal functions defined as the distance from some ``origin'' is \emph{half true} for the log-moment function.
That is, the analogue of Theorem \ref{go away from average} is true to one side of the origin, but because the log-moment function is not even we can not apply symmetry to the other side.

\begin{theorem}\label{half will go away from the origin}
Let $t > 0$ and $X$ an optimal Lipschitz function for $G$.
If $X(u) < \mathbb{E}(X)$, then there exists a $v \in N(u)$ such that $X(v) = X(u) + 1$.
\end{theorem}
\begin{proof}
We begin with a similar set-up as Theorem \ref{go away from average}.
Assume that for all $v \in N(u)$ we have that $X(v) \leq X(u) + 1 - \delta$.
For $0 < -\epsilon < \delta$, define $X_\epsilon$ to be the variable $X_\epsilon(w) = X(w)$ when $w \neq u$ and $X_\epsilon(u) = X(u) + \epsilon$.
By our assumption, $X_\epsilon$ is Lipschitz over $G$.
To prove the theorem, we will show that $e^{L_{X_\epsilon}(t)} > e^{L_{X}(t)}$ when $0 < -\epsilon$.

Let $p = 1/n$ be the probability of $u$, and so 
$$e^{L_{X_\epsilon}(t)} = e^{-p t \epsilon}e^{L_{X}(t)} + \left(e^{t\epsilon(1-p)} - e^{-p t \epsilon}\right) pe^{t(X(u) - \mathbb{E}(x))}.$$
The direct calculation gives us
$$ \frac{d}{d \epsilon} e^{L_{X_\epsilon}(t)} = -pt e^{-p t \epsilon}e^{L_{X}(t)} + \left(t(1-p)e^{t\epsilon(1-p)} + pt e^{-p t \epsilon}\right) pe^{t(X(u) - \mathbb{E}(x))}.$$
Because $1+y\leq e^y$ for all $y \in \mathbb{R}$, we have that 
$$  e^{L_{X}(t)} 	 \geq   \mathbb{E} (1 + t(X(v) - \mathbb{E}(X)))  =  1 .$$
Combine the previous two equations with our assumption $X(u) < \mathbb{E}(X)$ to see that 
$$ \frac{d}{d \epsilon} e^{L_{X_\epsilon}(t)} < p(1-p)te^{-p t \epsilon}(e^{\epsilon t} - 1). $$
On the domain $\epsilon < 0$ we have that $\frac{d}{d \epsilon} e^{L_{X_\epsilon}(t)} < 0$, so the theorem follows.
\end{proof}

Following the same proof as Theorem \ref{defining the origin} (but without the symmetry), we have half of the analogous result.

\begin{corollary} \label{subgaussian origin}
If $X$ is an optimal function and $X(u) < 0$, then $\nu_X-X(u) = d(u, \mathbb{O}(X))$.
\end{corollary}

To see that ``half'' of the result is best possible, consider the following example.
Let $G$ be the graph with vertices $\{v_1, v_2, v_3, v_4, w_1, w_2\}$ and edges $\{v_1v_2, v_2v_3, v_3v_4, w_1v_2, w_2v_3\}$.
We will compare two extremal Lipschitz functions on $G$.
Let $X_1(v_i) = X_2(v_i) = i$, $X_1(w_1) = X_2(w_1) = 1$, $X_1(w_2) = 4$, and $X_2(w_2) = 2$.
We have that $X_1$ is variance-optimal, but Mathematica was able to show that $L_{X_1}(t) < L_{X_2}(t)$ when $t \geq 3$.

We produce one more result, which will be used later.
We omit the proof, as it is obvious.

\begin{claim} \label{subgraphs}
If $G \subseteq H$, then $c(G) \geq c(H)$ and $\sigma_G \geq \sigma_H$.
\end{claim}

\subsection{Tightness for isoperimetric inequalities} \label{variance-optimal and isoperimetric section}
Recall that $S_{r,X}$ is the subset of vertices in $G^n$ defined as $\{(a_1, a_2, \ldots, a_n) : \sum_{i=1}^n X(a_i) \leq r\}$ for function $X:V(G)\rightarrow \mathbb{R}$.
When the function is implied, we drop the $X$ and set $S_r = S_{r,X}$.

The isoperimetric problem for product graphs $G^n$ and a number $d$ is to identify a set $S$ of a least half of the vertices of $G^n$ such that $i_{G,d,n} = |\{u : d(u,S) > d\}|$ is minimized. 
Alon, Boppana, and Spencer \cite{ABS} proved that $i_{G,d,n}$ decays exponentially as $d$ grows when $\sqrt{n} \ll d \ll n$ with a rate that relies on $c^2(G)$.
Let us now explore Conjecture \ref{is variance optimal correct}, which is that there exists a stronger relationship between $i_{G,d,n}$ and $c^2(G)$---that the extremal set $S$ can be determined by $X$ when $X$ is variance-optimal.

We will show that the conjecture is not true.
The issue is that any variable $X$ is forced to represent a (potentially complex) graph to a one-dimensional space (see Theorem \ref{go away from average} and surrounding discussion).
The conjecture has been seen to hold when the underlying graph has a distinctly one-dimensional topology.
It even holds for graphs with multiple dimensional-topology; for example the Euclidean grid is the standard two-dimensional graph, and since the Euclidean grid is $P_k \square P_k$ the conjecture holds because it holds for $P_k$.
But the conjecture begins to fail when the graph has some central point, but the rest of the graph can not be neatly labeled as ``up'' or ``down'' from that central point.

To build up an understanding that counters our original intuition that the conjecture is true, let us explore the assumption $m(X) \leq 0 = \mathbb{E}(X)$ that is in Theorem \ref{variance optimal is mostly right} but missing from Conjecture \ref{is variance optimal correct}.
In these early examples we only explore how $|B_k(S_{m(X)})|$ grows with $k$ for $n=1$.
We will do this with two examples: in the first having $m(X) < \mathbb{E}(X)$ is the ``correct thing to do,'' while the opposite is true in the second example.

Our first example is the unbalanced tripod $T_{k,k,2k}$, which is one vertex attached to $3$ hairs with $k+1,k+1,2k+1$ vertices respectively.
Formally, we define 
$$V(T_{k,k,2k}) = \{r\} \cup \{x_{i} : 1 \leq i \leq k\} \cup \{y_i: 1 \leq i \leq k\}\cup \{z_i: 1 \leq i \leq 2k\}$$
 and 
$$E(T_{k,k,2k}) = \{rx_{1},ry_{1},rz_{1}\} \cup \{x_{i}x_{i+1}:1 \leq i \leq k-1\} \cup $$
$$\cup \{y_{i}y_{i+1}:1 \leq i \leq k-1\}\cup \{z_{i}z_{i+1}:1 \leq i \leq 2k-1\}.$$
We have built up enough results to determine the variance-optimal functions of this tree.

\begin{example} \label{tripod}
For sufficiently large $k$, an extremal function for $T_{k,k,2k}$ is $X$, where $X(r) =0$, $X(x_i) = X(y_i) = -i$, and $X(z_i) = i$.
In this case, $m(x) = 0$ and $\mathbb{E}(X) > 0$.
There exists a permutation $\psi$ over $V(T_{k,k,2k})$ such that $\psi(S_{0,X}) = S_{0,-X}$ and $\psi(B_d(S_{0,X})) \subsetneq B_d(S_{0,-X})$.
\end{example}
\begin{proof}
Because it is significantly simpler, let us begin by proving the second part and later prove that $X$ is variance-optimal.
Consider the permutation $\psi$ over $V(T_{k,k,2k})$ such that $\psi^2 = 1$, $\psi(r) = r$, $\psi(x_i) = z_{2i-1}$, and $\psi(y_i) = z_{2i}$.
By construction, $\psi(S_{0,X}) = S_{0,-X}$, $|B_d(S_{0,X})| = 2k + 1 + d$, and $|B_d(S_{0,-X})| = 2k + 1 + 2d$.

Consider a Lipschitz variable $Y$ over $V(T_{k,k,2k})$ such that $\var(Y) \geq \var(X)$.
Note that $\mathbb{E}(X) \approx k/4$, and so 
$$ \var(X) \approx \frac{1}{4k} \left(\sum_{i=1}^{7k/4}i^2 + 2 \sum_{i=1}^{5k/4}i^2 - \sum_{i=1}^{k/4}i^2\right) \approx \frac{37}{48}k^2 \geq 0.77k^2. $$  
By translation invariance, let us assume that $Y(r) = 0$.
We will refer to the three hairs as the $x$-hair, the $y$-hair, and the $z$-hair.
By (\ref{unimodular hairs}), each hair can be broken into one or two monotone sequences.
The example can thus be analyzed by exhausting through a handful of cases.
The analysis will be simpler by first showing that at most $2$ and at least $1$ hair is monotone.

By definition, if a $w$-hair is monotone, then there exists a $\delta_w \in \{-1,1\}$ such that $Y(w_i) = i\delta_w$.
We have that $Y \in \{X,-X\}$ when all three hairs are monotone and $\delta_x = \delta_y \neq \delta_z$.
A simple case analysis shows that these values for $\delta_x, \delta_y, \delta_z$ maximize the variance when all three hairs are monotonic.
Some hair must have a minimal element of $Y$, a second hair must have a maximal element of $Y$, and the third hair satisfies the assumptions of Lemma \ref{hairs go one way}.
The conclusion of Lemma \ref{hairs go one way} is that the third hair is monotonic.

\textbf{Case 1:} the $z$-hair is monotone.
By symmetry, assume $\delta_z = 1$ and that the minimum element of $Y$ over the $x$-hair is at most the minimum element of $Y$ over the $y$-hair.
By Lemma \ref{hairs go one way}, we have that the $y$-hair is monotone.
If $\delta_y = -1$, then Lemma \ref{hairs go one way} applies to the $x$-hair, and all $3$ hairs are monotonic.
This is a contradiction, so $\delta_y = 1$ and the $x$-hair is not monotonic.
By considering monotone sequences as the extreme values of $Y(x_i)$, we see that $k/2 < \mathbb{E}(Y) < 3k/4$.

If the $x$-hair is not monotone, then Lemma \ref{hairs go one way} does not apply and for some $i$ we have $Y(x_i) < 0$.
Also, by Theorem \ref{go away from average}, there exists a $j$ with $Y(x_j) \geq \mathbb{E}(Y) > k/2$.
For both of these facts to be true, it must be that there exists an $\ell$ such that $Y(x_i) = -i$ for $1 \leq i \leq \ell$ and $Y(x_i) = -2\ell+i$ for $i \geq \ell$.
Moreover, $\ell < k/4$.
But then for all $w \in V(T_{k,k,2k}) - \{y_{5k/4+1}, \ldots, y_{2k}\}$ we have that $|Y(w) - \mathbb{E}(Y)| \leq 3k/4$.
So we see that 
$$\var(Y) \leq \frac{1}{4k+1}\left((4k+1 - 3k/4)(3k/4)^2 + \sum_{i=5k/4+1}^{2k} (i - k/2)^2\right)  $$
 $$ \approx \frac{1}{4k}\left(\frac{117}{64}k^3 + \frac{9}{8}k^3 - \frac{9}{64}k^3\right) =  \frac{171}{256}k^2 \leq 0.67k^2 < \var(X).$$

\textbf{Case 2:} the $z$-hair is not monotone.
Lemma \ref{hairs go one way} does not apply to the $z$-hair, so the $z$-hair contains a maximum or minimum element of $Y$.
By symmetry, let us assume that $Y(z_\ell)$ is the maximum value of $Y$ for some value of $\ell$.
Theorem \ref{go away from average} implies that there exists an $\ell' > \ell$ such that $Y(z_{\ell'}) < \mathbb{E}(Y)$.
Note that $\mathbb{E}(Z) \leq 3k/4$ for any Lipschitz variable $Z$ such that $Z(r) = 0$, and therefore $\ell \leq 11k/8$.

Some hair is monotone, so by symmetry the $x$-hair is monotone.

\textbf{Case 2.A:} $\delta_x = -1$.
Lemma \ref{hairs go one way} implies that the $y$-hair is monotone.
If $\delta_y = 1$, then the constraint that $Y(z_{\ell'}) < \mathbb{E}(Y)$ implies that $\ell/k \leq 1 + (\sqrt{10} - 3) \leq 1.163$ and $\mathbb{E}(Y) \leq 2(\sqrt{10} - 3)k \leq 0.325 k$.
On the other hand, $\delta_y = 1$ implies $\mathbb{E}(Y) \geq k/4$.
But then for all $w \in V(T_{k,k,2k}) - \{x_{k/2+1}, \ldots, x_k\}$ we have that $|Y(w) - \mathbb{E}(Y)| \leq 3k/4$.
Using a calculation similar to the end of case 1, we see that $\var(Y) < \var(X)$.

So assume $\delta_y = \delta_x = -1$.
The constraint that $Y(z_{\ell'}) < \mathbb{E}(Y)$ implies that $\ell < k$ and $\mathbb{E}(Y) < 0$.
But then 
$$ \var(Y) \approx \frac{1}{k}\sum_{i=1}^k i^2 \approx \frac{1}{3}k^2 < \var(X).$$

\textbf{Case 2.B:} $\delta_x = 1$.
It follows that $k/4 \leq \mathbb{E}(Y) \leq 3k/4$.
So for all $i$, we have that $|Y(x_i) - \mathbb{E}(y)| \leq 3k/4$.
If $\mathbb{E}(Y) \geq k/2$, then for all $i$ we have that $|Y(z_i) - \mathbb{E}(Y)| \leq 3k/4$.
If $\mathbb{E}(Y) \leq k/2$, then for all $i$ we have that $|Y(y_i) - \mathbb{E}(Y)| \leq 3k/4$.
So we arrive at a contradiction similar to the end of case 1.
\end{proof}

Our second example is a slight modification to the unbalanced tripod.

\begin{example} \label{tripod with star}
Fix a large odd $k$.
Let $G'$ be a single vertex with $k+3$ hairs: three of them have $k$ vertices and $k$ of them have $1$ vertex.
As before, let $r$ be the unique vertex with degree greater than $2$ and let the three hairs have vertices $x_i, y_i, z_i$; and now denote the vertices in the short hairs as $w_1, \ldots, w_k$.
Let $G$ be $G'$ plus edges $w_1z_2, w_1w_2, w_1w_3$.

The variable $X$ such that $\var(X) = c^2(G')$ satisfies $X(r) = 0$, $X(x_i) = X(y_i) = i$, $X(z_i) = -i$, and $X(w_i) = -1$ is variance-optimal.
As before, $m(x) = 0$ and $\mathbb{E}(X) > 0$.
We have that $|S_{-2,X}| = |S_{-(k+3)/2,-X}|$.
However, there exists a permutation $\psi$ of the vertex set such that $\psi(B_d(S_{-2,X})) \supsetneq B_d(S_{-(k+3)/2,-X})$.
\end{example}

\begin{proof}
First, let us prove that $X$ is variance optimal over $G'$.
By translation-invariance, let us assume $X(r) = 0$.
As in Example \ref{tripod}, one hair may have a maximal element of $X$, a second hair may have a minimal element of $X$, and Lemma \ref{hairs go one way} will apply to the third.
Because the three long hairs have equal length, if they are monotone then they contain an extremal value.
So Lemma \ref{hairs go one way} will apply to the $x$, $y$, and $z$-hairs.

By symmetry, we can assume $\delta_x = \delta_y = 1$ (by notation from Example \ref{tripod}), and so $\mathbb{E}(X) \geq k/8 > 1$.
Theorem \ref{go away from average} says that $X(w_i) = -1$ for all $i$.
We then have only two cases to check: when $\delta_z = 1$ and when $\delta_z = -1$, and a direct calculation gives that $X$ is variance-optimal when $\delta_z = -1$.

Now notice that $G' \subseteq G$, so by Claim \ref{subgraphs} we have that $c(G') \geq c(G)$.
But $X$ is Lipschitz over $G$, so $X$ is variance optimal over $G$.

Consider the permutation $\psi$ over $V(G)$ such that $\psi^2 = 1$, $\psi(r) = r$,
\begin{itemize}
	\item $\psi(x_{k-i}) = z_{k-2i}$ for $k-i \geq (k+1)/2$ (specifically note that $\psi(x_{(k+1)/2}) = z_1$), 
	\item $\psi(y_{k-i}) = z_{k-2i-1}$ for $k-i \geq (k+3)/2$,
	\item $\psi(y_{(k+1)/2}) = w_1$, 
	\item $\psi(y_{(x-1)/2}) = w_2$, $\psi(y_{(k-1)/2}) = w_3$, 
	\item $\psi(x_i) = w_{3+i}$ for $i \leq (k-3)/2$, and
	\item $\psi(y_i) = w_{(k+3)/2 + i}$ for $i \leq (k-3)/2$.
\end{itemize}

By construction $S_{-2,X} = \{z_2, z_3, \ldots, z_k\}$ and $S_{-(k+3)/2,-X} = \{x_{(k+3)/2}, \ldots, x_k, y_{(k+3)/2}, \ldots, y_k\}$. 
Both sets have $k-1$ vertices.
Moreover, we see that 
\begin{itemize}
	\item $B_1(S_{-2,X}) = S_{-2,X} \cup \{w_1, z_1\}$,
	\item $B_2(S_{-2,X}) = B_1(S_{-2,X}) \cup \{w_2, r\}$, 
	\item for $d \geq 3$, we have $V(G) - B_d(S_{-2,X}) = \{x_{d-2}, \ldots, x_k, y_{d-2}, \ldots, y_k\}$.
	\item for $d < (k+3)/2$ we have $B_d(S_{-(k+3)/2,-X}) = S_{-(k+3)/2,-X}) \cup \{ x_{(k+3 - 2d)/2}, \ldots, x_{(k+1)/2}, y_{(k+3 - 2d)/2}, \ldots, y_{(k+1)/2}\}$,
	\item for $d = (k+3)/2$ we have $B_d(S_{-(k+3)/2,-X}) = B_{d-1}(S_{-(k+3)/2,-X}) \cup \{r\}$, and 
	\item for $d > (k+3)/2$ we have $V(G) - B_d(S_{-(k+3)/2,-X}) =  \{z_{d - (k+1)/2} \ldots, z_k\}$.
\end{itemize}
So we have that $\psi(B_d(S_{-2,X})) \supseteq B_d(S_{-(k+3)/2,-X})$ for all $d$, and strict containment for $d > 2$.
\end{proof}

Our third example is similar to Examples \ref{tripod} and \ref{tripod with star}; it is a tree whose main tomography is two long paths whose endpoints are attached to a central vertex.
Similar to before, we will establish a permutation $\psi$ on our counterexample graph $G$ with variance-optimal function $X$ such that $\psi(B_d(S_{r,X})) \subsetneq B_d(S_{r,X})$.
The improvement of this example over Examples \ref{tripod} and \ref{tripod with star} is that this relation will hold for all $r$ and $d$, which implies that this relation holds when $\psi$ is tensored into a higher dimension as a permutation $\psi^n$ over $G^n$.

One distinction between $G$ and Examples \ref{tripod} and \ref{tripod with star} is that the two long paths are not hairs, but instead have many hairs of length $h$ attached to them.
In Theorem \ref{variance optimal is not isoperimetric} we only present $h=1$, but the result holds with greater discrepancy between $V(G^n) - B_d(S_{r,X})$ and $i_{G,d,n}$ for $h \leq O(1) \ll k$, where $k$ is the length of the paths.

\begin{theorem} \label{variance optimal is not isoperimetric}
Conjecture \ref{is variance optimal correct} is not true.
\end{theorem}
\begin{proof}
We consider the even-length caterpillar with one leg per body segment.
Formally, this graph is a path $u_1, u_2, \ldots, u_{2k}$ and a set of leaves $w_1, \ldots, w_{2k}$ such that $N(w_i) = \{u_i\}$.
When drawn, this graph resembles a hair comb.
Let $G$ denote this graph.

Suppose we have some Lipschitz variable $X$ on $G$.
Let $\pi$ be a permutation on $\{1, \ldots, 2k\}$ such that $X(u_{\pi(i)}) \leq X(u_{\pi(i+1)})$.
Let $X_\pi$ be such that $X_\pi(u_i) = X(u_{\pi(i)})$ and $X_\pi(w_i) = X(w_{\pi(i)})$.
Note that $X_\pi$ is also Lipschitz.
Moreover, if $X$ is variance-optimal, then for every $i$ we have that $|X(u_i) - X(u_{i+1})| = |X_\pi(u_i) - X_\pi(u_{i+1})| = 1$.
Because $X$ is variance-optimal and translation invariant, we may then assume that $X(u_i) = i$.
Theorem \ref{go away from average} implies that $X(w_i) = i-1$ if $i \leq k$ and $X(w_i) = i+1$ otherwise.

Let us write out in full detail the ordering on $V(G)$ imposed by $X$.
We do this by \emph{levels}, where level $i$ of function $Z:V(G)\rightarrow \mathbb{R}$ is the vertex set $Z^{-1}(i)$.
The levels of $X$ are from $0$ to $2k+1$ and they compose of 
\begin{itemize}
	\item level $0$ is $\{w_1\}$, 
	\item level $i$ for $1 \leq i \leq k-1$ is $\{w_{i+1}, u_i\}$, 
	\item level $k$ is $\{u_k\}$,
	\item level $k+1$ is $\{u_{k+1}\}$, 
	\item level $j$ for $k+2 \leq j \leq 2k$ is $\{w_{j-1}, u_{j}\}$, and 
	\item level $2k+1$ is $\{w_{2k}\}$.
\end{itemize}
Let us propose a different ordering of $V(G)$.
Let us call this ordering $Y : V(G) \rightarrow \mathbb{R}$ and the levels of $Y$ are composed of  
\begin{itemize}
	\item levels $0$ to $k$ are the same as for $X$, 
	\item level $k+1$ is $\{w_{k+1}\}$, 
	\item level $j$ for $k+2 \leq j \leq 2k$ is vertices $\{w_{j}, u_{j-1}\}$, and 
	\item level $2k+1$ is $\{u_{2k}\}$.
\end{itemize}
The levels have the same number of vertices for $X$ and for $Y$.

We claim that $Y$ is a better ordering than $X$.
Unfortunately $Y$ is not a $1$-Lipschitz function, as $Y(u_k) = k$ and $Y(u_{k+1}) = k+2$.
However, the set $S_{r,Y}$ is still a well-defined object.
We also have that $\mathbb{E}(X) = \mathbb{E}(Y) = m(X) = m(Y) = (2k+1)/2$.

If $r \leq (2k+1)/2$, then $S_{r,Y} = S_{r,X}$ as the levels from $0$ to $k$ are defined to be the same.
If $r \geq k+1$, then $|B_d(S_{r,X})| = |S_{r,X}| + 2d$.
If $r \geq k+1$ and $d \in \{1,2\}$, then $|B_d(S_{r,Y})| = |S_{r,Y}| + d$.
If $r \geq k+1$ and $d \geq 3$, then $|B_d(S_{r,Y})| = |S_{r,Y}| + 2d - 2$.

There exists a permutation $\psi$ over $V(G)$ such that $X(u) = Y(\psi(u))$ and $\psi(B_d(S_{r,X})) \supseteq B_d(S_{r,Y})$ with strict containment when $d > 0$ and $r \geq k+2$.
So if $\psi^n$ is a permutation ov $V(G^n)$ such that $\psi$ is applied to each coordinate, it follows that $\psi^n(B_d(S_{r,X})) \supseteq B_d(S_{r,Y})$, also with strict containment when $d > 0$ and $r \geq k+2$ (recall that $m(X^n) = n(2k+1)/2$).
\end{proof}


\section{Discrete Positive Curvature}

\subsection{Convex sets and iterated midpoints}\label{OV conj}
The notion of ``convex'' is undefined for discrete spaces, but Ollivier and Villani define it for this context to be the property that $\widehat{m}(S,S) \subseteq S$.
We will first give an example of convex sets $A,B$ in the hypercube $\mathbb{H}_{12}$ where $m_{1/2}(A,m_{1/2}(A,B))$ is much larger than $m_{1/4}(A,B)$.
It will be clear how this example generalizes to larger dimensions.
Then we will prove that our examples of $A,B$ are typical of all convex sets in the hypercube.

This section will be easier if we use the notation of the Boolean lattice, which is equivalent to the hypercube.
That is, the points of $\mathbb{H}_d$ are represented as the subsets of $\{1,2,\ldots, d\}$ and the distance between two points is the order of their symmetric difference.

\begin{example} \label{two convex sets}
Let $A,B \subset \mathbb{H}_{12}$, where $A = \{\alpha: \alpha \subseteq \{1,2,3,4\}\}$ and $B = \{\beta: \beta \supseteq \{1,2,\ldots,8\}\}$.
Each set $A$ and $B$ is a subspace of $H_{12}$ that is isometric to $H_{4}$ and therefore is convex.
We can directly calculate that if $\gamma \in m_{1/2}(A,B)$, then $4 \leq |\gamma| \leq 8$.
Also, if $\gamma \in m_{1/4}(A,B)$, then $2 \leq |\gamma| \leq 6$.
Now let us consider $m_{1/2}(A,m_{1/2}(A,B))$.
First off, the set $\phi=\{7,8,\ldots,12\}$ is a midpoint of $\emptyset \in A$ and $\{1,2,\ldots,12\} \in B$ and thus is in $m_{1/2}(A,B)$.
Now notice that $\zeta=\{1,2,3,4,8,9,10,11,12\}$ is a midpoint of $\{1,2,3,4\} \in A$ and $\phi$ and thus is in $m_{1/2}(A,m_{1/2}(A,B))$.
But $|\zeta| = 9$, which is too large to be in $m_{1/2}(A,B)$, much less $m_{1/4}(A,B)$!
\end{example}

To fully refute Ollivier and Villani's suggestion for approximating $m_{1/4}$, we need to show that this example is typical--as they only suggest that this method will ``probably'' work.
For this purpose, we show that \emph{every} convex closure of a set of points in the hypercube is the embedding of a smaller dimensional hypercube.

\begin{theorem}\label{convex closure is smaller hypercube}
If $S$ is a subset of the Boolean lattice such that $\widehat{m}(S,S) \subseteq S$, then there exists sets $\alpha, \beta$ such that 
$$S = \{\gamma: \alpha \subseteq \gamma \subseteq \beta\}.$$
\end{theorem}

\begin{proof}
First, it should be clear that $S$ has a unique maximal element, as otherwise a midpoint between the sets will have more elements.
A symmetric argument gives a unique minimal element.
Iterated applications of the midpoint argument gives every set in between the maximal and minimal element.
\end{proof}

We found the following consequences of this result to be interesting, as it illustrates how unusual the behavior can be for discrete spaces that appear nice and simple.

\begin{corollary} \label{closure of a ball}
The convex closure of any non-trivial ball in a hypercube is the whole space.
\end{corollary}
\begin{proof}
By symmetry, let us assume that the center of the ball is $\emptyset$.
Because the ball is non-trivial (in other words, the radius is positive), the ball contains the element $\{i\}$ for all $i \in \{1, \ldots, n\}$ as it is the minimum distance from $\emptyset$.
By Theorem \ref{convex closure is smaller hypercube}, the convex closure of the ball contains $\emptyset$ and $\cup_i \{i\}$ and everything in-between.
\end{proof}

\begin{corollary} \label{closure of subsets}
If $A,B$ are nonempty sets of vertices in the hypercube and $C$ is the convex closure of $\widehat{m}(A,B)$, then $A \cup B \subseteq C$.
\end{corollary}
\begin{proof}
For each $i \in \{1,\ldots,d\}$, there exists an automorphism $\phi_i$ of $\mathbb{H}_d$ such that $\phi_i(\alpha) = \alpha \cup \{i\}$ when $i \notin \alpha$ and $\phi_i(\alpha) = \alpha \setminus \{i\}$ otherwise.
Let $\alpha$ be an arbitrary fixed element of $A$, and pick some $\beta \in B$.
Let $\phi = \prod_{i \notin \alpha} \phi_i$, and let us consider the space $\mathbb{H}_d$ after $\phi$ has been applied.
By construction, $\phi(\alpha) = \{1,\ldots,n\}$.
For each $j$, there exists a $\gamma \in \widehat{m}(\alpha,\beta)$ such that $j \in \gamma$.
So by Theorem \ref{convex closure is smaller hypercube}, the convex closure of $\widehat{m}(\alpha,\beta)$ contains $\alpha$.
The corollary follows from symmetry.
\end{proof}

\subsection{The $\ell_0$, $\ell_1$, and $\ell_\infty$ metric} \label{go curve, go}

Let us note that there are natural constructions of the $\ell_0$, $\ell_1$, and $\ell_\infty$ metrics in graph theory.
Suppose $G$ is the product of spaces $H_1, H_2, \ldots, H_d$.
The $\ell_1$ metric is denoted $H_1 \square H_2 \square \cdots \square H_d$, as mentioned above.
The $\ell_0$ metric is isometric to  $K_{|H_1|} \square K_{|H_2|} \square \cdots \square K_{|H_d|}$.
Finally, the $\ell_\infty$ metric is denoted by the \emph{tensor product}, which is $H_1 \boxtimes H_2 \boxtimes \cdots \boxtimes H_d$.

\begin{theorem} \label{l0 metric}
Suppose $S,T \subseteq G = K_{r_1} \square K_{r_2} \square \cdots \square K_{r_d}$ such that $d_*(S,T) \geq \delta d$.
Let $\rho$ be such that $|1/2-\rho| < \epsilon$ and $C = \delta^2 - 16\ln(2)\epsilon > 0$.
Under these conditions, for large $d$ we have that 
$$|\widehat{m}_{\rho}(S,T)| > \sqrt{|S||T|} e^{C d(1/2+o(1))}.$$
\end{theorem}
\begin{proof}
We use shorthand $(A,B)_r = \{(a,b):a \in A, b \in B, d(a,b) = r\}$.
We claim that for each $d \geq r \geq \delta d$ that 
$$ |(\widehat{m}_{\rho}(S,T), \widehat{m}_{\rho}(S,T))_r| > |(S,T)_r|e^{\left( (\delta d/r)^2 - 16\ln(2)\epsilon\right)r(1 + o(1))}. $$
The claim will prove the theorem, as 
$$\left( (\delta d/r)^2 - 16\ln(2)\epsilon\right)r = d\left( \delta^2 \frac{d}{r} - 16\ln(2)\epsilon \frac{r}{d}\right) \geq C d $$
 for $0 < r \leq d$ by the definition of $C$.

We define $C_{\ell_1, \ldots, \ell_k}^{(r)} = \{A \subseteq \{0,1\}^r: |A| = \ell_i, 1\leq i \leq k\}$ as before.
We will define a map 
$$\phi: (S,T)_r \times C^{(r)}_{\lfloor \rho r \rfloor, \lceil (1-\rho) r\rceil} \rightarrow (\widehat{m}_{\rho}(S,T), \widehat{m}_{\rho}(S,T))_r.$$
Moreover, we will demonstrate that this map satisfies for any $(m^{(1)}, m^{(2)}) \in (\widehat{m}_{\rho}(S,T), \widehat{m}_{\rho}(S,T))_r$ with $d \geq r \geq \delta d$, the property 
$$|\{x: \phi(x) = (m^{(1)},m^{(2)})\}| < |C^{(r)}_{\lfloor \rho r \rfloor, \lceil (1-\rho) r\rceil}|e^{-\left( (\delta d/r)^2 - 16\ln(2)\epsilon\right)r(1 + o(1))},$$ 
which proves the claim.

For $a \in  G = K_{r_1} \square K_{r_2} \square \cdots \square K_{r_d}$, we will use notation $a = (a_1, a_2, \ldots, a_j, \ldots, a_d)$ to refer to individual coordinates.
For a pair $(s,t) \in (S,T)_r$, let $i_1 < \cdots < i_r$ be the coordinates where $s_i \neq t_i$.
Let $\pi \in  C^{(r)}_{\lfloor \rho r \rfloor, \lceil (1-\rho) r\rceil}$ and $\phi( (s,t), \pi) = (m^{(1)},m^{(2)})$.
We define 
$$ m^{(1)}_j = \left\{ \begin{array}{ll} s_j,  & \mbox{if } s_j = t_j \\ s_j, & \mbox{if } j = i_k, k \in \pi \\ t_j, & \mbox{if } j = i_k, k \notin \pi \end{array} \right.   , m^{(2)}_j = \left\{ \begin{array}{ll} s_j, & \mbox{if } s_j = t_j \\ t_j, & \mbox{if } j = i_k, k \in \pi \\ s_j, & \mbox{if } j = i_k, k \notin \pi \end{array} \right. .$$
It is clear that 
\begin{itemize}
	\item $m^{(1)}$ and $m^{(2)}$ are weighted midpoints between $s$ and $t$, 
	\item that $d(s,m^{(2)}) = |\pi| \in \{\lfloor \rho r \rfloor, \lceil (1-\rho) r\rceil\}$, and 
	\item $d(s,m^{(1)}) = r-|\pi| \in \{\lfloor \rho r \rfloor, \lceil (1-\rho) r\rceil\}$.
\end{itemize}
Thus, $m^{(1)},m^{(2)} \in \widehat{m}_{\rho}(S,T)$.
Moreover, $d(m^{(1)},m^{(2)}) = r$, and so $(m^{(1)},m^{(2)}) \in (\widehat{m}_{\rho}(S,T),\widehat{m}_{\rho}(S,T))_r$.

All that remains is to prove that for any fixed $(m^{(1)},m^{(2)}) \in (\widehat{m}_{\rho}(S,T),\widehat{m}_{\rho}(S,T))_r$, we have that $\phi^{-1}(m^{(1)},m^{(2)})$ is not too large.
Let $\phi'$ be defined in the same way as $\phi$, but on the extended domain $(G,G)_r \times C^{(r)}_{\lfloor \rho r \rfloor, \lceil (1-\rho) r\rceil}$.
By the definition of $\phi$, it is a simple calculation to see that if $\phi((s,t),\pi) = (m^{(1)},m^{(2)})$, then $\phi'((m^{(1)},m^{(2)}),\pi) = (s,t)$.
So for any fixed $\pi$, we have that $|\{(s,t): \phi((s,t),\pi) = (m^{(1)},m^{(2)})\}| \leq 1$, and it equals $1$ if and only if $\phi'((m^{(1)},m^{(2)}),\pi) \in S \times T$.
The claim will then follow when we demonstrate next that for 
$$|C^{(r)}_{\lfloor \rho r \rfloor, \lceil (1-\rho) r\rceil}|\left(1-e^{-\left( (\delta d/r)^2 - 16\ln(2)\epsilon\right)r(1 + o(1))}\right)$$
 values of $\pi$, we have that $\phi'((m^{(1)},m^{(2)}),\pi) \notin S \times T$.

Recall that we are assuming $(m^{(1)},m^{(2)})$ is fixed.
We define sets 
$$S' =\{ \pi \in C^{(r)}_{\lfloor \rho r \rfloor, \lceil (1-\rho) r\rceil} : \phi'((m^{(1)},m^{(2)}),\pi) \in S \times G\}$$
and 
$$T' =\{ \pi \in C^{(r)}_{\lfloor \rho r \rfloor, \lceil (1-\rho) r\rceil} : \phi'((m^{(1)},m^{(2)}),\pi) \in T \times G\}.$$
Suppose $\phi'((m^{(1)},m^{(2)}),\pi) = (a,b)$ and $\phi'((m^{(1)},m^{(2)}),\pi_*) = (a_*,b_*)$.
By definition of $\phi'$, it should be clear that $d(\pi,\pi _*) = d(a,a_*) = d(b,b_*)$.
Therefore $\min_{\pi \in S', \pi_* \in T'}d(\pi, \pi') \geq d_*(S,T) = \delta d \geq \delta r$.
The proof then follows from applying the second part of Theorem \ref{concentration on levels of Boolean lattice}.
\end{proof}

The $\ell_0$, $\ell_1$, and $\ell_\infty$ metrics are special in that $m(\{a\},\{b\})$ is not unique in standard space.
The positive curvature in $\ell_0$ metrics can entirely be attributed to the fact that $m(\{a\},\{b\})$ grows exponentially with $d(a,b)$.
But what about $\ell_1$ and $\ell_\infty$?
The problem with these metrics is that $m(\{a\},\{b\})$ \emph{might} be unique.
Suppose our space $G$ is the $n$-dimensional product of subspace $H$, and let points in $G$ be denoted by tuples $x = (x_1, \ldots, x_d)$, $y = (y_1, \ldots, y_d)$.
If $G$ is equipped with $\ell_1$ metric, then $m(\{x\},\{y\})$ may be unique (depending on $H$) if $x_i = y_i$ for all $i \geq 2$.
If $G$ is equipped with $\ell_\infty$ metric, then $m(\{x\},\{y\})$ may be unique (depending on $H$) if $x_i = x_j$ and $y_i = y_j$ for all $i,j$.

Essentially the issue with the $\ell_1$ and $\ell_\infty$ metrics is that there exists a convex embedding of $H$ in $G$.
This is similar to Ruzsa \cite{R} and Gardner and Gronchi's \cite{GG} problems for establishing a discrete analogue of the Brunn-Minkowski inequality: there is a degenerate case where the sets lie in a smaller dimension.
This leads to several natural open questions.

\begin{open} \label{other metrics}
Suppose $G$ is the $d$-dimensional product of space $H$ equipped with the $\ell_1$ metric.
\begin{enumerate}
	\item If we force the sets $A$ and $B$ to have dimension $d$ in a manner similar to Gardner and Gronchi's work, how do the sets of midpoints grow with the distance between $A$ and $B$? (The same question can be asked in the $\ell_\infty$ metric).
	\item If the distance between $A$ and $B$ is asymptotically bigger than the diameter of $H$, do we see an exponential growth in the number of midpoints between sets similar to the growth when equipped with $\ell_0$?
\end{enumerate}
\end{open}

There are several initial statements that can be said of the questions in \ref{other metrics}.
First, the $d$-dimensional hypercube embeds into $G$ for any $H$ when equipped with the $\ell_1$ metric, and this establishes a best-case-possible because Ollivier and Villani's result is known to be asymptotically tight.
Second, if we are equipped with the $\ell_1$ metric, then a proof similar to that of Theorem \ref{l0 metric} will show that the set of points within distance $d_H$ of the midpoints of $A$ and $B$ is at least $\sqrt{|A||B|} e^{\frac{1}{16d}\left(\frac{d(A,B)}{d_H}\right)^2}$, where $d_H$ is the diameter of $H$.
This is because $G$ has, with error term $d_H$, a ``rough geometry'' equivalent to that of metric $\ell_0$.
However, if $H$ has strong negative curvature, then it is not clear that the midpoints themselves will be large.
Finally, note that the second question is similar to the first---with the change that we are requiring the set $A-B$ to have large dimension instead of $A$ and $B$ themselves, which may be sufficient due to the non-uniqueness of midpoints.
A similar question involving the $\ell_\infty$ metric will require the opposite assumption: that $A,B$ coexist in a small dimensional subspace.

\subsection{Catalog of displacement convexity definitions}\label{disp conv sec}

There are several different versions of displacement convexity.
The general intuition is that when supply $M_A:V(G) \rightarrow \mathbb{R}_{\geq0}$ is optimally routed to demand $M_B:V(G) \rightarrow \mathbb{R}_{\geq0}$ across a time span $t \in [a, b]$, then the mid-transit goods in a positively curved space at time $t$ should be more ``spread out'' than the convex combination of $\frac{t-a}{b-a}$ of the spread of $M_A$ and $\frac{b-t}{b-a}$ of the spread of $M_B$.

Formally, the functions $M_A, M_B$ are normalized to probability spaces $\mu_A, \mu_B$.
A function $\tau: V(G) \times V(G) \rightarrow \mathbb{R}_{\geq 0}$ is a \emph{transportation} of the goods from the supply to the demand if $\mu_A(\alpha) = \sum_\beta \tau (\alpha, \beta)$ for all $\alpha$ and $\mu_B(\beta) = \sum_\alpha \tau (\alpha, \beta)$ for all $\beta$.
The Wasserstein cost of a transportation $\tau$ is $W^1(\tau) = \sum_{\alpha,\beta} \tau(\alpha, \beta) d(\alpha, \beta)$, and the Wassertsein cost of order two is $W^2(\tau) = \sum_{\alpha,\beta} \tau(\alpha, \beta) d(\alpha, \beta)^2$.
The Wasserstein distance between $\mu_A$ and $\mu_B$ is $W(\mu_A, \mu_B) = \inf_\tau W^1(\tau)$.
The generalization where $M_A, M_B$ are not normalized is also known as the Earth Mover's Distance.
An \emph{optimal} transport $\tau$ minimizes $W^2$.
For finite spaces such a minimum clearly exists.

For each $\alpha, \beta \in A \times B$, we place a probability distribution $Q_{\alpha, \beta}$ on the geodesics from $\alpha$ to $\beta$.
For a geodesic $P$ that starts at $\alpha$ and ends at $\beta$, let $P(t)$ for $0 \leq t \leq 1$ be the point on $P$ whose distance from $\alpha$ is $t d(\alpha, \beta)$.
For a fixed optimal transport $\tau$ and probability distributions $Q_{\alpha, \beta}$, we define probability distribution $\mu_t$ for $0 \leq t \leq 1$ to be 
$$ \mu_t(\gamma) = \sum_{\alpha \in A, \beta \in B} \tau(\alpha, \beta) \sum_{\gamma = P(t)} Q_{\alpha,\beta}(P) . $$
Distance interpolation is an optimal transportation $\tau$ combined with a uniform distribution applied to each $Q_{\alpha, \beta}$.
For the discrete analogue, a point $P(t)$ would be $P( \frac{\lfloor t d(\alpha, \beta) \rfloor }{ d(\alpha, \beta) } )$ and $P( \frac{ \lceil t d(\alpha, \beta) \rceil }{d(\alpha, \beta)} )$.

For probability space $\mu$, let $H(\mu)$ denote its entropy.
The formal notion of mid-transit goods being more spread out than the supply and demand is that 
\begin{equation} \label{entropy eqn}
H(\mu_t) \geq t H(\mu_A) + (1-t)H(\mu_B) + \frac{k}{2}t(1-t) W^2(\mu_A, \mu_B),
\end{equation}
 where $K$ is a nonnegative number that represents the strength of the curvature of the space $G$.
The distinction between the different versions of displacement convexity come from the distinction between the words ``any'' and ``every'' when it comes to multiple optimal transportations or multiple geodesics between points (events that occur rarely in geometric geodesic spaces, but which played a large role in Section \ref{go curve, go}).
We outline several definitions below.

\begin{itemize}
	\item (Strong displacement convexity) For any optimal transportation $\tau$ and geodesic choices $Q_{\alpha, \beta}$, (\ref{entropy eqn}) holds.
	\item (Sort-of-strong displacement convexity) (\ref{entropy eqn}) holds for distance interpolation on any optimal transportation.
	\item (Sort-of-weak displacement convexity) (\ref{entropy eqn}) holds for distance interpolation on some optimal transportation.
	\item (Weak displacement convexity) There exists optimal transportation $\tau$ and geodesic choices $Q_{\alpha, \beta}$, possibly dependent on $\mu_A, \mu_B$, such that (\ref{entropy eqn}) holds.
\end{itemize}

Strong displacement convexity is presented as the ``normal'' version in \cite{OV}.
That the inequality should hold for any transportation or geodesic is also the condition in \cite{EM}, although they are not working with displacement convexity.
The conditions of sort-of-strong curvature are presented as the definition of distance interpolation in Theorem 7.21 of \cite{V}, conditions that imply uniqueness of geodesics and optimal transportations are used later (such as Definition 16.5, which is set up by chapters 9 and 10).
Gozlan, Roberto, Samson, and Tetali \cite{GRST} work with sort-of-weak displacement convexity, but they use transportations that minimize $W^1$ instead of $W^2$ and ``midpoints'' take mass Gaussian distributed across a geodesic rather than being a point mass.
Weak displacement convexity is presented in chapter 29 of \cite{V}, and in \cite{LV}, although it is just called ``displacement convexity'' except in the bottom of page 906.
Weak displacement convexity is also used by Bonciocat and Sturm (see equation (2.1) of \cite{BS}) in their work on approximate midpoints.

If we do not specify a type of displacement convexity, then the statement holds for all four forms of displacement convexity.
In the following sections we will prove statements about specific types of displacement convexity, but first we will prove several statements that hold for general displacement convexity.
To do so, we modify our definition of midpoints.
For vertices $a,b$, let the midpoints between them $\widetilde{m}(a,b)$ be as follows:\\
(A) if $d(a,b)$ is even, then $\widetilde{m}(a,b) = \{u \in V(G) : 2d(u,a) = 2d(u,b) = d(a,b)\}$, and \\
(B) if $d(a,b)$ is odd,  then $\widetilde{m}(a,b) = \{(u,v) \in E(G) : 2d(u,a)+1 = 2d(v,b)+1 = d(a,b)\}$. \\

We use several ideas from continuous transportation theory in our work.
Some of those ideas translate into similar techniques in the discrete setting; while other ideas translate into performing the \emph{opposite} technique.
Specifically, we will use the relationship between cyclical monotonicity (Definition 5.1 of \cite{V}) and optimal transportations as before.
On the other hand, we reverse the Monge-Mather shortening principle (chapter 8 of \cite{V}), because after Lemma \ref{induct via partition} we will only be interested in transportations whose entire support is one class in the transitive closure of the relation $(a,b)~(a',b')$ when some $(a,b)$-geodesic intersects some $(a',b')$-geodesic.

A transportation $\tau$ is \emph{cyclically monotone} if for every sequence of pairs $(a_1, b_1), \ldots, (a_k, b_k)$ we have that $\sum_{i=1}^k d(a_i, b_i)^2 \leq d(a_1, b_k)^2 + \sum_{i=1}^{k-1}d(a_{i+1},b_i)^2$.
It is easy to confirm that a transportation is optimal if and only if it is cyclically monotone, as this is the same technique as finding alternating cycles when looking for maximum weight matchings in a bipartite graph.

Recall that in the proof to Theorem \ref{l0 metric} we split $A \times B$ to classes $(A \times B)_r = \{(a,b): d(a,b) = r\}$.
We will do this again using the following lemmas.

\begin{lemma} \label{separate into distances}
If $\tau$ is an optimal transportation for probability functions $\mu_A, \mu_B$ over graph $G$, $\tau(a_1, b_1) \tau(a_2, b_2) > 0$, and $\widetilde{m}(a_1,b_1) \cap \widetilde{m}(a_1,b_1) \neq \emptyset$, then $d(a_1,b_1) = d(a_2, b_2)$.
\end{lemma}
\begin{proof}
By contradiction, assume that $d(a_1,b_1) < d(a_2, b_2)$.
The proof to the lemma is similar when $d(a_1,b_1)$ is odd or even, so we will assume $d(a_1,b_1)$ is odd and allow the reader to handle the even case.
Let $d(a_1, b_1) = 2k+1$ and $d(a_2, b_2) = 2 \ell + 1 \geq 2k+3$.
Let $(u,v) \in \widetilde{m}(a_1,b_1) \cap \widetilde{m}(a_1,b_1)$, and let $d(a_1, u) = k$, $d(b_1, v) = k$.
Repeated use of the triangle inequality implies $d(a_1, b_2), d(a_2, b_1) \leq k + \ell + 1$.
We can directly calculate that $d(a_1, b_1)^2 + d(a_2, b_2)^2 - d(a_1, b_2)^2 - d(a_2, b_1)^2 \geq 2(k-\ell)^2 > 0$, which contradicts that $\tau$ is cyclically monotone.
\end{proof}

We now give a formal definition for a partition of a transportation.
Lemma \ref{separate into distances} implies that each distinct set of distances involved in a transportation can be used to construct a partition.

\begin{definition} \label{partition up a transportation}
Let $\tau$ be a transportation from $\mu_A$ to $\mu_B$.
Let $E_1, E_2$ be nonempty sets such that $E_1 \cup E_2 = \{(a,b):\tau(a,b)>0\}$ and if $(a_1,b_1) \in E_1$, $(a_1, b_2) \in E_2$, then $\widetilde{m}(a_1, b_1) \cap \widetilde{m}(a_2, b_2) = \emptyset$.
Then $E_1, E_2$ form a \emph{partition} of $\tau$.
Let $\eta_1 = \sum_{(a,b) \in E_1} \tau(a,b)$, define probability functions $\mu_{A,1}(x) = \eta_1^{-1}\sum_{(x,b) \in E_1}\tau(x,b)$ and $\mu_{B,1}(y) = \eta_1^{-1}\sum_{(a,y) \in E_1}\tau(a,y)$, and let us define transportation $\tau_1: \mu_{A,1} \rightarrow \mu_{B,1}$ as $\tau(a,b) = \eta_1^{-1} \tau(a,b)$ if $(a,b) \in E_1$ and $\tau(a,b) = 0$ otherwise.
Define $\eta_2, \mu_{A,2}, \mu_{B,2}, \tau_2$ similarly.
\end{definition}

We remark that $\tau = \eta_1 \tau_1 + \eta_2 \tau_2$, and so for all $k$ we have that $W^k(\tau) = \eta_1 W^k(\tau_1) + \eta_2 W^k(\tau_2)$.
The following claim is then obvious, and we omit the proof.

\begin{claim}\label{optimal broken up is optimal}
We use the notation of Definition \ref{partition up a transportation}.
If $\tau$ is an optimal transportation, then so are $\tau_1$ and $\tau_2$.
If $\tau$ is an optimal transportation and $\tau_1': \mu_{A,1} \rightarrow \mu_{B,1}$, $\tau_2': \mu_{A,2} \rightarrow \mu_{B,2}$ are optimal transportations, then $\tau' = \eta_1 \tau_1' + \eta_2 \tau_2'$ is also an optimal transportation from $\mu_A$ to $\mu_B$.
\end{claim}

In the next lemma we show that it suffices to prove curvature exists among transportations without partitions in order to prove that curvature exists in general.
We follow this with another lemma, which describes transportations without partitions.

\begin{lemma} \label{induct via partition}
Let $\tau$ be an optimal transportation from $\mu_A$ to $\mu_B$ with a partition as defined in Definition \ref{partition up a transportation}.
Let $\mu_C, \mu_{C,1}, \mu_{C,2}$ be the probability measures for the midpoints using transportations $\tau, \tau_1, \tau_2$, respectively.
If for $i \in \{1,2\}$ there exists fixed values $k, C$ and convex function $f$ such that
$$ S(\mu_{C,i}) \geq C(S(\mu_{A,i}) + S(\mu_{B,i})) + f\left(W^{k}(\mu_{A_i}, \mu_{B_i}) \right),$$
then 
$$ S(\mu_{C}) \geq C(S(\mu_{A}) + S(\mu_{B})) + f\left(W^{k}(\mu_A, \mu_B) \right).$$
\end{lemma}
\begin{proof}
By the definition of of partition, the support of $\mu_{C,1}$ and the support of $\mu_{C,2}$ is disjoint, and therefore $\mu_C = \eta_1 \mu_{C,1} + \eta_2 \mu_{C,2}$.
By the formula for entropy, this implies that $S(\mu_C) = \eta_1 (S(\mu_{C,1}) - \ln(\eta_1)) + \eta_2 (S(\mu_{C,2}) - \ln(\eta_2))$.
By the convexity of entropy, the equality for $\mu_C$ is an upper bound: $S(\mu_A) \leq \eta_1 (S(\mu_{A,1}) - \ln(\eta_1)) + \eta_2 (S(\mu_{A,2}) - \ln(\eta_2))$ and $S(\mu_B) \leq \eta_1 (S(\mu_{B,1}) - \ln(\eta_1)) + \eta_2 (S(\mu_{B,2}) - \ln(\eta_2))$.
So we are left with 
$$ S(\mu_{C}) \geq C(S(\mu_{A}) + S(\mu_{B})) + \eta_1 f\left(W^k(\mu_{A,1}, \mu_{B,1})\right) + \eta_2 f\left(W^k(\mu_{A,2}, \mu_{B,2})\right), $$
and thus the lemma follows from the convexity of $f$ and Claim \ref{optimal broken up is optimal}.
\end{proof}

\begin{lemma} \label{everybody is large}
Let $\tau$ be an optimal transportation from $\mu_A$ to $\mu_B$ that does not have a partition.  
Let $A$ be the support of $\mu_A$ and $B$ be the support of $\mu_B$.\\
(1) There exists a constant $D$ such that if $\tau(a,b) > 0$ then $d(a,b) = D$.\\
(2) For the $D$ in part (1) and any $a \in A$ and $b \in B$, we have $d(a,b) \geq D$. \\
(3) For the $D$ in part (1) and for all $k$, $W^k(\mu_A, \mu_B) = D^k$.
\end{lemma}
\begin{proof}
(1)This is a restatement of Lemma \ref{separate into distances}. 

(2) For a fixed transportation $\tau$, we define a graph $J_\tau'$, with vertex set $S_\tau = \{(a,b) : \tau(a,b) > 0\}$ and edge set $E(J_\tau') = \{(a_1, b_1)(a_2,b_2): \widetilde{m}(a_1, b_1) \cap \widetilde{m}(a_2, b_2) \neq \emptyset\}$.
Because $\tau$ has no partition, $J_\tau'$ is a connected graph.

By way of contradiction, assume that $d(x,y) < D$ for some fixed pair $x \in A, y \in B$.  
Let $(a_1,b_1)$ be such that $\tau(a_1, b_1) > 0$ and $x = a_1$ and let $(a_*,b_*)$ be such that $\tau(a_*, b_*) > 0$ and $y = b_*$.
Because $J_\tau'$ is connected, there exists a path $(a_1, b_1), (a_2, b_2), \ldots, (a_\ell, b_\ell) = (a_*, b_*)$.
By (1), we have that $d(a_i, b_i) = D$ for all $i$.
By the triangle inequality and the assumption $\widetilde{m}(a_i, b_i) \cap \widetilde{m}(a_{i+1}, b_{i+1}) \neq \emptyset$, we have that $d(b_i, a_{i+1}) \leq D$.
This is a contradiction, as $\tau$ is not cyclically monotone.

(3) By (2), we see that $W^k(\mu_A, \mu_B) \geq D^k$.  By (1), we see that $W^k(\mu_A, \mu_B) \leq W^k(\tau) = D^k$.
\end{proof}

Let us close this subsection by remarking that finding an optimal transportation for a given $\mu_A, \mu_B$ with a fixed metric $d$ is a problem in linear programming, where the support of each of $\mu_A$ and $\mu_B$ appear as a unique constraint (elements in the intersection of the supports of $\mu_A$ and $\mu_B$ thus show up as two constraints).
The solution to the dual problem is $(\phi,\psi)$, where $\phi$ is a function over the support of $\mu_A$ and $\psi$ is a function over the support of $\mu_B$.
The dual problem is named as the ``dual Kantorovich problem'' by Villani in chapter 5 of \cite{V}.

\begin{remark} \label{constant valued functions}
A consequence of Lemma \ref{everybody is large} applied to Remark 5.13 of \cite{V} is that the optimal transportations between $\mu_A$ and $\mu_B$ will have no partitions if and only if the solutions to the dual problem only have constant valued functions $\phi, \psi$.
\end{remark}

\subsection{Displacement convexity and the hypercube}\label{curve of hypercube sec}

Our first result is that the hypercube does not have positive or nonnegative sort-of-strong displacement convexity.

\begin{example} \label{negative curvature}
The $d$-dimensional hypercube for $d \geq 10$ contains vertex sets $A$ and $B$ such that when $\mu_A$ and $\mu_B$ are uniform probability measures with support on $A$ and $B$ there exists an optimal transport $\tau :A \times B \rightarrow \mathbb{R}$ such that the entropy of the probability space 
$$ \mu_C(\gamma) = \sum_{\alpha \in A, \beta \in B, \gamma \in \widehat{m}(\alpha,\beta)} \frac{\tau(\alpha, \beta)}{|\widehat{m}(\alpha,\beta)|}  $$
is less than the average of the entropies of $\mu_A$ and $\mu_B$.
\end{example}
\begin{proof}
We return to the notation of the Boolean lattice introduced in Section \ref{OV conj}.
Let $A = \{\{1\}, \{2\}, \ldots, \{d'\}\}$ and $B = \{\{d'+1\}, \{d'+2\}, \ldots, \{2d'\}\}$ for $2 d' \leq d$.
Every vertex in $A$ is distance $2$ from any vertex in $B$, so any transportation function $\tau$ is optimal.
We choose the transportation function such that $\tau(\{i\},\{d'+i\}) = \mu_A(\{i\}) = \mu_B(\{d'+i\}) = \frac{1}{d'}$ for each $1 \leq i \leq d'$.
So $\mu_C(\{i, d'+i\}) = \frac{1}{2d'}$, $\mu_C(\emptyset) = 1/2$, and $\mu_C$ is $0$ otherwise.
The entropy of $\mu_A$ and $\mu_B$ are each $\ln(d')$, and the entropy of $\mu_C$ is $\ln(d')/2 + \ln(2)$.
The example thus holds when $d' > 4$.
\end{proof}

Now we will present a result that is progress towards showing that the hypercube does have positive sort-of-weak displacement convexity.
Let us assume for the rest of the section that all transportations do not have a partition as in Definition \ref{partition up a transportation}.
Let $A$ be the support of $\mu_A$ and $B$ be the support of $\mu_B$; the lack of a partition implies that $A \cap B = \emptyset$.

Recall the map $\phi: (S,T)_r \times C^{(r)}_{\lfloor r/2 \rfloor, \lceil r/2\rceil} \rightarrow (\widehat{m}(S,T), \widehat{m}(S,T))_r$ defined in the proof to Theorem \ref{l0 metric}.
We generalize it as follows: let $C_R = \widetilde{m}(\{\emptyset\}, \{\{1,2,\ldots,R\}\})$ using the notation of the Boolean lattice, and let 
$\widetilde{\phi}: A \times B \times C_R \rightarrow \widetilde{m}(A,B) \times \widehat{m}(A,B)$.
If $R$ is even, then $C_R = C^{(R)}_{R/2}$ and $\phi = \widetilde{\phi}$.
If $R$ is odd, then elements of $C_R$ are edges $xy$, where $|x| = (R-1)/2, |y| = (R+1)/2$.
We then define 
$$ \widetilde{\phi}(a, b, (x,y)) = ((\phi(a,b,x),\phi(a,b,y)), (\phi(a,b,\{1,\ldots,R\}\setminus y),\phi(a,b,\{1,\ldots,R\}\setminus x))). $$

We have one advantage on the hypercube when working with an optimal transportation rather than the transportation $\tau(a,b) = \mu_A(a)\mu_B(b)$, which is used for Brunn-Minkowski curvature: in the following Lemma we show that $\widetilde{\phi}$ is injective.

\begin{lemma}\label{injective map from endpoints to midpoints}
If $(\alpha_1, \beta_1), (\alpha_2, \beta_2) \in A \times B$ and $\tau(\alpha_1, \beta_1) \tau(\alpha_2, \beta_2) > 0$ for optimal transport $\tau$, then $\phi(\alpha_1, \beta_1, \pi_1) \neq \phi(\alpha_2, \beta_2, \pi_2)$ for any $\pi_1, \pi_2 \in C_R$.
\end{lemma}
\begin{proof}
First, let us assume that $R$ is even, and so $C_R$ is a collection of vertices.
Recall that $d(\pi_1,\pi_2) = d(\alpha_1,\alpha_2) = d(\beta_1,\beta_2)$.
Also, that if $\pi'$ is fixed, then $\phi_{\pi'}(s,t) = \phi(s,t,\pi')$ is injective.
Therefore $d(\pi_1, \pi_2) > 0$.
Notice that $d(\alpha_1, \beta_2) = d(\alpha_2, \beta_1) = R - d(\pi_1,\pi_2) < R$.
This contradicts Lemma \ref{everybody is large}.

Now suppose that $R$ is odd, so that $\pi_1 = (x_1, y_1)$ and $\pi_2 = (x_2, y_2)$.
We have that $d(\alpha_1,\alpha_2) = d(\beta_1,\beta_2) = d(x_1, y_2) = d(y_1, x_2)$, and the same argument follows.
\end{proof}

As we saw in Example \ref{negative curvature}, the midpoints will not spread out for every transportation between sets.
Each transportation function is a probability measure over the space $V(G) \times V(G)$, and we will work with the probability measure that maximizes $S(\tau)$ (while still satisfying the conditions of being an optimal transportation).  

\begin{lemma}\label{uniform across a midpoint}
Fix probability distributions $\mu_A, \mu_B$, and let $\tau$ maximize the entropy among optimal transportations from $\mu_A$ to $\mu_B$.
Assume that $\tau$ has no partition.
For a fixed element $X \in V(G) \cup E(G)$, let $\{(a_1,b_1), \ldots, (a_k, b_k)\}$ be the pairs of points such that $\tau(a_i, b_i) > 0$ and $X \in \widetilde{m}(a_i, b_i)$.
Possibly the list $a_1, \ldots, a_k$ has repeats, and so may $b_1, \ldots, b_k$.
Then there exists a $t$ such that $d(a_i, b_j) = t$ for all $i,j$, and $X \in \widetilde{m}(a_i, b_j)$ for all $i,j$.
Moreover, there exists a function $\tau_X:G \rightarrow \mathbb{R}$ such that $\tau(a_i, b_j) = \tau_X(a_i) \tau_X(b_j)$ for all $i,j$.
\end{lemma}
\begin{proof}
That there exists a $t$ such that $d(a_i, b_i) = t$ for all $i$ is a restatement of Lemma \ref{separate into distances}.
By the triangle inequality, we have that $d(a_i, b_j) \leq t$ for all $i,j$.
By Lemma \ref{everybody is large} $d(a_i, b_j) \geq t$.
Moreover, the sharpness of the triangle inequality implies that $X \in \widetilde{m}(a_i, b_j)$ for all $i,j$.

Let $\tau_X^\circ(a_i) = \sum_j \tau(a_i, b_j)$, $\tau_X^\circ(b_j) = \sum_i \tau(a_i, b_j)$, and $T_X = \sum_i \tau_X^\circ(a_i) = \sum_j \tau_X^\circ (b_j)$.
Let $\tau^*(a_i,b_j) = \tau_X^\circ(a_i) \tau_X^\circ(b_j) / T_X$ and $\tau^*(u,v) = \tau(u,v)$ otherwise.
By construction, $\tau^*$ is a transportation from $\mu_A$ to $\mu_B$, and $W^2(\tau^*) = W^2(\tau)$, so $\tau^*$ is also optimal.
We see that $\tau^*$ can be thought of as the product of distributions $\tau_X^\circ(\{a_1, \ldots, a_k\}) \times \tau_X^\circ(\{b_1, \ldots, a_{k'}\})$, which is known to maximize entropy over product spaces by the independence inequality.
Therefore $S(\tau^*) \geq S(\tau)$, with equality only when $\tau^* = \tau$.
By construction, $\tau^*$ satisfies the conclusion of the lemma ($\tau_X = \tau_X^\circ / \sqrt{T_X}$), and so the proof concludes by the condition that $\tau$ had maximum entropy among optimal transportations.
\end{proof}

Before we prove our theorem about the entropy of midpoints, let us recall a few facts about entropy.
Let $\mu$ be a probability measure over the product space $X = X_1 \times X_2 \times \cdots X_k$.
The entropy of $\mu$ is defined as $S(\mu) = \sum_y \mu(y) \ln\left( \frac1{\mu(y)} \right)$.
We are interested in how the entropy acts when we restrict some coordinates of $X$.
For $T \subseteq \{1,2,\ldots,k\}$, let $X_T = \prod_{j \in T} X_j$.
For $T' \subseteq T$, $z' \in X_{T'}$, and $z \in X_{T}$, we use the notation $z \in z'$ to denote the situation where $z'_j = z_j$ for all $j \in T'$.
Let $\mu_T$ be the projection of $\mu$ into $X_{T}$, which is equivalently the probability distribution over $X_{T}$ such that $\mu_T(z) = \sum_{y \in z} \mu(y)$ (because $\mu = \mu_{\{1,2,\ldots,k\}}$).
For $T' \subseteq T$, the conditional entropy is 
$$S(\mu_{T} | \mu_{T'}) = S(\mu_{T}) - S(\mu_{T'}) =  \sum_{z' \in X_{T'}} \sum_{z \in z'} \mu_{T}(z)\ln\left(\frac{\mu_{T'}(z')}{\mu_T(z)} \right) $$
$$ = \sum_{y \in X} \sum_{y \in z \in z'} \mu(y) \ln\left(\frac{\mu_{T'}(z')}{\mu_T(z)} \right). $$
From the definitions, when $z' \in z$ we have $\mu_T(z) \leq \mu_{T'}(z')$, and so the conditional entropy is always nonnegative.
The name comes from the fact that if $z' \in X_{T'}$ is fixed and $y_{z'} \in X_T$ is a random variable conditioned on $y_{z'} \in z'$, then $S(\mu_{T} | \mu_{T'}) = \mathbb{E}_{z'} S(y_{z'})$.
From the formula, we see that if $T'' \subseteq T' \subseteq T$, then $S(\mu_{T'} | \mu_{T''}) = S(\mu_{T}|\mu_{T''}) - S(\mu_{T}|\mu_{T'})$.

We define a probability distribution $\zeta$ over $A \times B \times C_R \times M \times M$, where $\zeta(a,b,c,m,m') = \tau(a,b)/|C_R|$ if $\widetilde{\phi}(a, b, c) = (m, m')$, and $\zeta(a,b,c,m,m') = 0$ otherwise.
We will study how the entropy of $\zeta$ changes as we project onto specific coordinates.
This will be clearer if we use the following abuse of notation: let $A$ represent the first coordinate, $B$ represent the second coordinate, $C_R$ represent the third coordinate, $M$ represent the fourth coordinate, and $M'$ represent the fifth coordinate.
So $\zeta_{\{A,C_R,M\}}$ is $\zeta$ projected onto the first, third, and fourth coordinates.
By construction, $\zeta_{\{A\}} = \mu_A$, $\zeta_{\{B\}} = \mu_B$, $\zeta_{\{M\}} = \zeta_{\{M'\}} = \mu_C$, $\tau = \zeta_{\{A,B\}}$, and $\zeta_{\{C_R\}}$ is a uniform distribution.
By symmetry, for any $S \subseteq \{A,B,C_R\}$ we have that $\zeta_{S \cup \{M\}}$ is isomorphic to $\zeta_{S \cup \{M'\}}$.
Because $\widetilde{\phi}$ is injective, each of $\zeta_{\{M,M'\}}, \zeta_{\{A,B,C_R\}}, \zeta_{\{A,B,M\}}$ is isomorphic to $\zeta$.

We will use the following technical statement.

\begin{claim}\label{weird entropy inequality}
Let $\zeta$ be defined as above for optimal transportation $\tau$ with maximum entropy that has no partition.
We have 
$$ S(\zeta_{\{A,M\}}|\zeta_{\{A\}}) - S(\zeta | \zeta_{\{B,M\}}) = S(\mu_C) - S(\mu_A), $$
$$ S(\zeta_{\{B,M\}}|\zeta_{\{B\}}) - S(\zeta | \zeta_{\{A,M\}}) = S(\mu_C) - S(\mu_B), $$
$$S(\zeta | \zeta_{\{A,M\}}) + S(\zeta | \zeta_{\{B,M\}}) = S(\zeta|\zeta_{\{M\}}),$$
$$ S(\zeta_{\{A,M\}}|\zeta_{\{A\}}), S(\zeta_{\{B,M\}}|\zeta_{\{B\}}) \geq \ln(|C_R|).$$
\end{claim}
\begin{proof}
By the injectivity of $\widetilde{\phi}$, the sets $(m,m')$ for a fixed $m$ is in bijection with the sets $(a,b)$ such that $m \in \widetilde{m}(a, b)$ and $\tau(a,b)>0$.
By Lemma \ref{uniform across a midpoint}, we have $\zeta_{\{A,M\}}(a,m) = \sum_{i}\frac{\delta_m(a)\delta_m(b_i)}{|C_R|}$ and 
\begin{eqnarray*}
\mu_A(a)	& = &	\sum_{b \in B} \tau(a,b) \\
		& = &	\sum_{b \in B} \frac{1}{|C_R|} \sum_{\phi(a,b,c) = (m,m')} \tau(a,b) \\
		& = &	\sum_{\widetilde{\phi}(a,b,c) = (m,m')} \frac{\delta_m(a) \delta_m(b) }{|C_R|}. 
\end{eqnarray*}

By the definition of conditional entropy and Lemma \ref{uniform across a midpoint}, we have that 
\begin{eqnarray*}
 S(\zeta_{\{A,M\}}|\zeta_{\{A\}}) & = & \sum_{a,m} \zeta_{\{A,M\}}(a,m) \ln \left(\frac{\zeta_{\{A\}}(a)}{\zeta_{\{A,M\}}(a,m)} \right)\\
			& = & \sum_{a,m} \left(\sum_{i}\frac{\delta_m(a)\delta_m(b_i)}{|C_R|}\right) \ln \left(\frac{\mu_A(a)|C_R|}{\sum_{i'}\delta_m(a)\delta_m(b_{i'})}\right) \\
		& = & \sum_{a,m,b} \frac{\delta_m(a)\delta_m(b)}{|C_R|}\ln \left(\frac{\mu_A(a)|C_R|}{\delta_m(a)\sum_{i'}\delta_m(b_{i'})}\right),
\end{eqnarray*}
and 
$$ S(\zeta | \zeta_{\{B,M\}}) = \sum_{\widetilde{\phi}(a,b,c) = (m,m')} \frac{\delta_m(a)\delta_m(b)}{|C_R|} \ln \left(\frac{\sum_{i'}\delta_m(a_{i'})}{\delta_m(a)} \right).$$
So 
\begin{eqnarray*}
S(\zeta_{\{A,M\}}|\zeta_{\{A\}}) - S(\zeta | \zeta_{\{A,M\}}) & = & \sum_{\widetilde{\phi}(a,b,c) = (m,m')} \frac{\delta_m(a)\delta_m(b)}{|C_R|} \ln \left(\frac{\mu_A(a)|C_R|}{\sum_{i'}\delta_m(a_{i'})\sum_{j'}\delta_m(b_{j'})} \right) \\
		& = & S(\mu_C) - S(\mu_A).
\end{eqnarray*}

The second equality follows from a symmetric argument.
The third equality follows a similar argument, with 
$$S(\zeta | \zeta_{\{M\}}) = \sum_{m \in M} \sum_{a_i, b_j} \frac{\delta_m(a_i) \delta_m(b_j)}{|C_R|} \ln\left(\frac{\left(\sum_{i'} \delta_m(a_{i'})\right)\left(\sum_{j'} \delta_m(b_{j'})\right)|C_R|}{\delta_m(a_i) \delta_m(b_j)} \right)$$
and 
$$S(\zeta | \zeta_{\{A,M\}}) = \sum_{a,b,m}\frac{\delta_m(a)\delta_m(b)}{|C_R|}\ln \left(\frac{\sum_{i}\delta_m(a_{i})}{\delta_m(a)}\right).$$

For the final inequality, we first remark that $\mu_A(a) \geq \delta_{m^*}(a) \sum_b \delta_{m^*}(b)$ for any fixed $m^*$.
So
\begin{eqnarray*}
S(\zeta_{\{A,M\}}|\zeta_{\{A\}})	& = &	\sum_{\widetilde{\phi}(a,b,c) = (m,m')} \frac{\delta_m(a) \delta_m(b) }{|C_R|} \ln\left(\frac{|C_R|\mu_A(a)}{\delta_m(a)\sum_{b'}\delta_m(b')} \right)\\
					& \geq & \sum_{\widetilde{\phi}(a,b,c) = (m,m')} \frac{\delta_m(a) \delta_m(b) }{|C_R|} \ln\left(|C_R| \right)\\
					& \geq & \ln(|C_R|).
\end{eqnarray*}
The result for $S(\zeta_{\{B,M\}}|\zeta_{\{B\}})$ follows symmetrically.
\end{proof}

Now we are prepared to prove the theorem.

\begin{theorem} \label{weak curvature for hypercube}
Let $\mu_A, \mu_B$ be probability distributions over the discrete hypercube.
Let $\tau$ be an optimal transportation from $\mu_A$ to $\mu_B$ that maximizes entropy, and let $\mu_C$ be the probability distribution over $\widetilde{m}(\mu_A, \mu_B)$ as calculated by $\tau$ with distance interpolation.
If $d(a,b) = R$ for some constant $R$ whenever $\tau(a,b) > 0$, then 
$$ S(\mu_C) \geq \frac13 (S(\mu_A) + S(\mu_B) + 2 \ln(|C_R|)).$$
\end{theorem}
\begin{proof}
Summing the first two equalities from Claim \ref{weird entropy inequality} and simplifying by the definition of conditional entropy, we have that 
$$ S(\zeta_{\{A,M\}}) + S(\zeta_{\{B,M\}}) - S(\zeta) = S(\mu_C). $$
The last inequality of Claim \ref{weird entropy inequality} implies that $S(\zeta_{\{A,M\}}) \geq S(\mu_A) + \ln(|C_R|)$, so 
\begin{equation}\label{sort of weak curve ineq 1}
S(\mu_C) \geq S(\mu_A) + S(\mu_B) - S(\zeta) + 2\ln(|C_R|).
\end{equation}
Now $S(\zeta) = S(\zeta_{\{M,M'\}}) \leq S(\zeta_{\{M\}}) + S(\zeta_{\{M'\}}),$ so 
\begin{equation}\label{sort of weak curve ineq 2}
S(\mu_C) \geq \frac{1}{2} S(\zeta).
\end{equation}
When we combine (\ref{sort of weak curve ineq 1}) and (\ref{sort of weak curve ineq 2}), we see that
$$S(\mu_c) \geq \max \{ S(\mu_A) + S(\mu_B) - S(\zeta) + 2 \ln(|C_R|), S(\zeta)/2\}.$$
The right hand side is minimized when $S(\zeta) = \frac23 (S(\mu_A) + S(\mu_B) + 2 \ln(|C_R|))$, and the resulting value is the statement of the theorem.
\end{proof}

\begin{theorem}\label{almost curved}
Let $\mu_A, \mu_B$ be probability distributions over the $d$-dimensional discrete hypercube.
There exists a transportation $\tau$ that when combined with distance interpolation produces a probability distribution $\mu_C$ over the midpoints such that 
$$ S(\mu_C) \geq \frac13 S(\mu_A) + \frac13 S(\mu_B) + \frac{2}{5d^3} (W^2(\mu_A, \mu_B))^2 - \frac{2}{3}. $$
\end{theorem}
\begin{proof}
It is easy to calculate that $|C_{2i}| = {2i \choose i} \approx 2^{2i} i^{-0.5}$ and that $|C_{2i+i}| = {2i +1 \choose i} (i+1) \approx 2^{2i+1} i ^{0.5}$.
We used computer software to calculate $\ln(2) \geq 0.69$ and to confirm that $\ln(|C_R|) \geq 0.6 R - 1$.
So Lemma \ref{induct via partition} applied to Theorem \ref{weak curvature for hypercube} provides that 
$$ S(\mu_C) \geq \frac13 (S(\mu_A) +  S(\mu_B)-2) + 0.4 W^1(\mu_A, \mu_B).$$
For any integers $\ell \leq \ell'$, $k \leq k'$, space $\Omega$ and probability measures $\mu_A, \mu_B$ we have that $W^\ell(\mu_A, \mu_B) \geq W^{\ell'}(\mu_A, \mu_B) \diam(\Omega)^{\ell-\ell'}$ and $(W^\ell(\mu_A, \mu_B))^k \geq (W^{\ell}(\mu_A, \mu_B))^{k'} \diam(\Omega)^{\ell(k-k')}$.
\end{proof}

Let us finish this section by returning to the Brunn-Minkowski curvature of the hypercube.

\begin{proposition} \label{hypercube can do better}
There exists $\delta>0$ and $D$ such that the $d$-dimensional hypercube for $d > D$ has Brunn-Minkowski curvature at least $\frac{1}{2d}(1+\delta)$.
\end{proposition}
\begin{proof}
First, let us consider the situation where $d(S,T) \geq (1-\epsilon_1)d$ for some $\epsilon_1 > 0$.
Note that for any $(a,b) \in (S,T)_r$, we have that 
\begin{equation}\label{many midpoints for large spaces}
 \widehat{m}(S,T) \geq \widehat{m}(a,b) \approx {r \choose r/2} \approx 2^r / \sqrt{r}.
\end{equation}
On the other hand, because $1 - \epsilon_1 > 1/2$, we have that $|S|,|T| \leq 2^d e^{d(1/2-\epsilon_1)^2/2}$.
So if $\epsilon_1 = 0.01$, then $r \geq 99d/100$, and for sufficiently large $d$ and sufficiently small $\delta$ we get that 
$$ |\widehat{m}(S,T)| \geq 2^{d(0.99 - o(1))} \geq 2^{d(1 - 0.49^2\ln(2)/2)} e^{\frac{d}{16}(1+\delta)} \geq \sqrt{|S||T|} e^{\frac{d(S,T)^2}{8}\frac{1}{2d}(1+\delta)}. $$

So now assume that $d(S,T) < (1-\epsilon_1)d$.
The version of the claim in the proof to Theorem \ref{l0 metric} that appears in \cite{OV} is that $|(\widehat{m}(S,T), \widehat{m}(S,T))_r| > |(S,T)_r| e^{d(S,T)^2/(8r)}$.
This inequality is weakest when $r$ is largest, so an improvement on this inequality for $r \geq (1-\epsilon_2)d$ for some fixed $\epsilon_2 > 0$ is sufficient for an improvement on the final result.
Fix a $c \in C^{(r)}_{\lfloor r/2 \rfloor, \lceil r/2\rceil}$, and define a function $\psi: A \times B \rightarrow \widehat{m}(S,T) \times C^{(d)}_{r}$ such that the first coordinate of $\psi(a,b)$ is $\phi(a,b,c)$ and the second coordinate is a vector over $\{0,1\}^d$ such that the set of coordinates with nonzero entries is the symmetric difference between $a, b$ as elements of the Boolean lattice.
By the binary nature of the discrete hypercube, $\psi$ is invertible.
But when $r \geq (1-\epsilon_2)d$ we have that $|C^{(d)}_{r}| \leq 2^de^{-(1/2-\epsilon_2)^2d/2}$ and for any $(a,b) \in (S,T)_r$ we have $|\widehat{m}(S,T)| \geq |\widehat{m}(a,b)| \geq {r \choose r/2} \approx 2^r/\sqrt{r}$.
So the stronger bound comes from 
$$ |(\widehat{m}(S,T), \widehat{m}(S,T))_r| \geq \frac{2^r/\sqrt{r}}{2^de^{-(1/2-\epsilon_2)^2d/2}}|\widehat{m}(S,T) \times C^{(d)}_{r}| \geq 2^{r-d}e^{d(1/2-\epsilon_2)^2/2}|(S,T)_r|.$$
By setting $\delta,\epsilon_2$ as a function of $\epsilon_1$, we have that $2^{-\epsilon_2d}e^{d(1-2\epsilon_2)^2/8} > e^{(1+\delta)d(1-\epsilon_1)^2/(8(1-\epsilon_2))}$.
\end{proof}

\subsection{Strong displacement convexity}\label{strong curve sec}

\begin{theorem} \label{char strong curve}
If $G$ is a graph with nonnegative strong displacement convexity, then $G$ is a path, a cycle, a complete graph, or a complete graph minus an edge, and $G$ has strong displacement convexity $0$.
\end{theorem}

First we will discuss the curvature of paths and cycles, and then we will show that no other graph can have non-negative curvature.

\begin{lemma}
The path and the cycle have strong displacement convexity $0$.
\end{lemma}
\begin{proof}
We assume that our optimal transportation $\tau$ has no partition, and then use Lemma \ref{induct via partition} to handle the other case.
But Lemma \ref{everybody is large} can only be satisfied on the path or cycle for point masses for $\mu_A, \mu_B$.
But then $S(\mu_C) \geq S(\tau) \geq \max\{S(\mu_A),S(\mu_B)\}$.
\end{proof}

\begin{lemma} \label{missing one edge at most}
If $G$ is a graph with nonnegative strong displacement convexity, then for any $v \in V(G)$ and $u_1, u_2, w_1, w_2 \in N(v)$ such that $(u_1, u_2) \neq (w_1, w_2)$ we have that $u_1u_2 \in E(G)$ or $w_1w_2 \in E(G)$.
\end{lemma}
\begin{proof}
By way of contradiction, suppose that $u_1, u_2, w_1, w_2 \in N(v)$, $u_1 \neq w_1$, and $u_1u_2,w_1w_2 \notin E(G)$.
Then let $\mu_A$ be the uniform distribution over $\{u_1, w_1\}$, and let $\mu_B$ be the uniform distribution over $\{u_2, w_2\}$ (which may be one or two points).
There exists an optimal transportation from $\mu_A$ to $\mu_B$ such that the only midpoint is $v$.
\end{proof}

\textit{Proof of Theorem \ref{char strong curve}.}
If $G$ is connected and every vertex has degree $1$ or $2$, then $G$ is a path or a cycle.
So suppose $v$ is incident with at least $3$ edges.
If $V(G) \subseteq N(v) \cup \{v\}$, then by Lemma \ref{missing one edge at most}, $G$ is either the complete graph or the complete graph minus an edge.
So assume there exists a $w$ such that $d(v,w) = 2$; equivalently $vw \notin E(G)$ and $N(v) \cap N(w) \neq \emptyset$.

If $N(v) \setminus N(w) \neq \emptyset$, then because $d(v)\geq 3$ and by Lemma \ref{missing one edge at most} there exists an $xy \in E(G)$ such that $x \in  N(v) \setminus N(w)$ and $y \in N(v) \cap N(w)$.
But this contradicts Lemma \ref{missing one edge at most}, because $N(y)$ is missing edges $vw$ and $xw$.
So assume that $N(v) \subseteq N(w)$, and by symmetry this implies $N(v) = N(w)$.

Now we claim that $\{v,w\} \cup N(v) = V(G)$.
If this is not true, then there exists a $x \in N(v)$ and a $y \notin \{v,w\} \cup N(v)$ such that $xy \in E(G)$.
But this contradicts Lemma \ref{missing one edge at most}, because $N(x)$ is missing edges $yv$ and $yw$.

So $N(v) = N(w) = V(G) - \{v,w\}$ and $vw \notin E(G)$.
If $G$ is not the complete graph minus an edge, then there exists a $xy \notin E(G)$.
By Lemma \ref{missing one edge at most}, $E(G)$ is all pairs of points except $vw$ and $xy$.
Because $d(v) \geq 3$, there exists a $z \in N(v) \setminus \{x,y\}$.
But this contradicts Lemma \ref{missing one edge at most}, because $N(z)$ is missing edges $xy$ and $vw$.
\hfill $\square$

\subsection{Other graphs}\label{and now to my field}
A typical graph will not have curvature for the same reasons that the hypercube does not: by examining the neighborhood of a single vertex.
The $-2/3$ in Theorem \ref{almost curved} is necessary for transportations $\tau$ with $W^2(\tau) = 1$ due to the fact that $\ln(|C_1|) = 0$.
However, we may be able to prove a ``rough geometry'' version of curvature.
That is, the curvature equation may exist with a small error term that accounts for transportations with small Wasserstein distances.
We establish a first result towards this goal by showing that expander graphs have some flavor of positive curvature.

\begin{theorem}
Let $G$ be a graph such that every vertex has degree $d$ and $\lambda \ll 1$ is the second largest eigenvalue of the normalized adjacency matrix.
Let $S,T$ be vertex sets such that $|S| = |T|$.
If $|S||T|e^{d(S,T)\ln(1+\lambda)} > \lambda n^2/4(1 + O(1))$, then $|m(S,T)| \geq \theta(\lambda n)$.
\end{theorem}
\begin{proof}
We define $S^i = \{v : d(v,S) \leq i\}$.
For disjoint vertex sets $X,Y$, $e(X,Y)$ denote the number of edges with an endpoint in each of $X$ and $Y$.
We will make use of two theorems from spectral graph theory.
The first is Corollary 5.5 of \cite{C}, which states that $e(S^i, S^{i+1}-S^i) \geq d|S^i|(1-\lambda)\frac{n-|S^i|}{n}$.
The second is the Expander Mixing Lemma (see Theorem 5.1 of \cite{C}), which states that 
$$ \left| e(X,Y) - d |X||Y|/n\right| \leq \lambda d \sqrt{|X||Y|}. $$ 

Each vertex is adjacent to $d$ edges,  and therefore $d|S^{i+1}-S^i| \geq e(S^i, S^{i+1}-S^i)$.
So if $2|S^i| \leq n$, then $|S^{i+1}| \geq |S^i|(1 + (1 - \lambda)/2)$ (note that $\lambda < 1$).
Let $S' = S^{d(S,T)/2}$ and $T' = T^{d(S,T)/2}$, and thus 
$$|S'||T'| \geq \min\{n^2/4, |S||T|e^{d(S,T)\ln(1.5 - \lambda/2)}\} $$
$$> \left(((1+\epsilon)\sqrt{\lambda} + \lambda)\frac{n}2\right)^2 = \lambda n^2/4(1 + O(1)).$$

Next, each vertex in $S'$ with a neighbor in $T'$ is a midpoint in $m(S,T)$.
Each vertex is adjacent to $d$ edges, and so $|m(S,T)| \geq |E(S',T')|/d$.
Therefore $|m(S,T)| \geq |S'||T'|/n - \lambda \sqrt{|S'||T'|} = \frac{1}{n}\left(\sqrt{|S'||T'|} - \sqrt{\lambda} n/2 \right)^2 - \lambda n/4$.
The theorem then follows from the bound on $|S'||T'|$.
\end{proof}

We are interested in determining whether the assumption on a fixed degree is necessary.
Using the normalized Laplacian, there exist generalizations of expander graphs for arbitrary degree distributions \cite{CG}.
However, it may be that even rough positive curvature will only exist when the degree distribution falls into a tight range.
A \emph{power-law} graph is a graph whose degree distribution approximately follows the distribution of the inverse of a polynomial (and hence are widely skewed).
Such graphs have become popular recently for their ability to model social, technical, and biological networks.
Models for such graphs include Kronecker graphs (used for Graph500).

\begin{conjecture}
It is impossible for a power-law graph to have ``big picture'' curvature.
\end{conjecture}

The conjecture would imply that positively curved networks will not be useful for social networks, as negative curvature has been.
On the other hand, it may still be useful to study engineered networks (for example, super computing clusters frequently use product topologies, such as the discrete torus).
Chung, Lu, and Vu \cite{CLV} studied the spectral properties of a random power-law graph.
We provide some evidence for our conjecture by studying random walks on arbitrary power-law graphs.

\begin{theorem} \label{random power law theorem 1}
Suppose we weight a path $P_{xy} = x,u_1, u_2, \ldots, u_k,y$ as $w(P_{xy}) = \prod_{i=1}^k \frac{1}{d(u_i)}$.
If we pick a random shortest path $P_{xy}$, where $P_{xy}$ is picked proportional to $w(P_{xy})$ across all pairs $x,y$, then the probability that $z$ is the midpoint of $P_{xy}$ is proportional to $d(z)$.
\end{theorem}
\begin{proof}
We consider a random process that is a random walk with edge teleportation.
Fix some arbitrary $0 < c < 1$.
We place some token on a vertex of the graph and will move that token according to a random walk with probability $c$ and according to edge teleportation with probability $1-c$.

We start the process with a stationary distribution on the vertices for a random walk; let the probability that our token is on $v$ be proportional to $d(v)$.
Suppose the token is on vertex $z$, and let us discuss how the token will move at the next step in the process.
With probability $c$ we move the token to a random neighbor of $z$.
Note that this step preserves the probability distribution; after a random walk step the probability that the token is on $z$ is still proportional to $d(z)$.
With probability $1-c$ we choose an edge uniformly at random and teleport the token to one of its endpoints chosen uniformly at random.
This too preserves the probability distribution; and so the probability that the token is on $z$ will be proportional to $d(z)$ as we iterate this process.

Let $z_0, z_1, \ldots$ be an infinite sequence of vertices such that the token is on vertex $z_i$ after $i$ iterations of our process.
Let $i_1, i_2, \ldots$ be the iterations such that the transition $z_{i_j-1}\rightarrow z_{i_j}$ involves edge teleportation for each $j$ and $z_{k-1}\rightarrow z_k$ involves a step in a random walk when $k \notin \{i_1, i_2, \ldots \}$.
We choose a random geodesic as follows: repeatedly pick some $j \in \mathbb{Z}$ until the path $P_j = z_{i_j}, z_{i_j+1}, \ldots, z_{i_{j+1}-1}$ is a geodesic.

The theorem will follow when we establish two facts: (1) the probability that the midpoint of $P_j$ is $z$ is proportional to $d(z)$ and (2) the probability that a fixed geodesic $P'$ is chosen is proportional to $w(P')$.
Part (1) is easy, as the probability of any vertex $z_i$ being $z$ is proportional to $d(z)$.
All that remains is (2).

Let $P_{xy} = x,u_1, u_2, \ldots, u_k,y$ be some fixed geodesic, and let us calculate the probability $q(P_{xy})$ that $P_{xy} = P_j$ for the first value of $j$ (in other words, without accounting for the fact that we throw out non-geodesic paths).
The probability that $z_{i_j} = x$ is $\frac{d(x)}{2|E(G)|}$.
The probability that $i_{j+1} = i_j + k+2$ is $c^{k+1}(1-c)$.
For shorthand, let us denote $u_0 = x$ and $u_{k+1} = y$.
Given that $z_{i_j + \ell}  = u_\ell$ and $i_{j+1} = i_j + k+2$, the probability that $z_{i_j + \ell + 1}  = u_{\ell+1}$ is $\frac{1}{d(u_\ell)}$.
Thus 
$$ q(P_{xy}) = \frac{d(x)}{2|E(G)|} c^{k+1}(1-c) \prod_{i=0}^k \frac{1}{d(u_i)} = \frac{1-c}{2|E(G)|} c^{k+1}\prod_{i=1}^k \frac{1}{d(u_i)}.$$

Now let us adjust the calculation of the probability of $P_{xy} = P_j$ to account for the fact that we will throw out non-geodesic paths.
The probability is clearly proportional to $q(P_{xy})$, when taken in comparison to all other paths.
When we consider the proportional values, the term $\frac{1-c}{2|E(G)|}$ cancels as it is a uniform constant.
Therefore the probability that $P_{xy} = P_j$ for the final value of $j$ is proportional to $c^{k+1}\prod_{i=1}^k \frac{1}{d(u_i)} = c^{k+1}w(P_{xy})$.
The theorem follows by considering the limit $c \rightarrow 1$.
\end{proof}

Now suppose we perform the same proof, but this time the probability of an edge teleportation from $z_i$ to $z_{i+1}$ is $\frac{1}{d(z_i)+1}$ rather than $c$.
The following theorem is the result of this modification, plus averaging the probability of path $P_{xy}$ with its reverse.

\begin{theorem}
Suppose we weight a path $P_{xy} = x=u_1, u_2, \ldots, u_k=y$ as $w(P_{xy}) = \left(d(x) + d(y) \right)\prod_{i=1}^k \frac{1}{d(u_i)+1}$.
If we pick a random shortest path $P_{xy}$, where $P_{xy}$ is picked proportional to $w(P_{xy})$ across all pairs $x,y$, then the probability that $z$ is the midpoint of $P_{xy}$ is proportional to $d(z)$.
\end{theorem}

\appendix

\section{Statements that were not used}
There does not seem to be a significant difference between weak displacement convexity and Brunn-Minkowski curvature.
We make this statement based on the fact that Lemmas \ref{separate into distances} and \ref{uniform across a midpoint} hold for any graph.
In our work towards Theorem \ref{weak curvature for hypercube}, we proved several lemmas that were not part of the final version of the proof.
However, the statements are interesting in their own right, and may be useful towards establishing a bound on curvature for other spaces.

The following claim is from when we were still working with $\widehat{m}$ instead of $\widetilde{m}$.

\begin{claim}
For probability distributions $\mu_A, \mu_B$, let $\tau$ be an optimal transportation.
Suppose $\tau(\alpha_1, \beta_1) \tau(\alpha_2, \beta_2) > 0$ and $\widehat{m}(\alpha_1, \beta_1) \cap \widehat{m}(\alpha_2, \beta_2) \neq \emptyset$.
Under these circumstances, $| d(\alpha_1, \beta_1) - d(\alpha_2, \beta_2) | \leq 2$.
\end{claim}
\begin{proof}
The proof by contradiction of Lemma \ref{separate into distances} will hold when $\lceil d(\alpha_1, \beta_1)/2 \rceil < \lfloor d(\alpha_2, \beta_2)/2 \rfloor$.
\end{proof}

The goal of the following lemma was to prove an analogue of Hall's Marriage theorem from the set of $(a,b)$ such that $\tau(a,b) > 0$ to the set of midpoints.
The outline was to prove that any collection $\{(a_1,b_1), \ldots, (a_k,b_k)\}$ had large sets $\{a_1, \ldots, a_k\}$ and $\{b_1, \ldots, b_k\}$, and therefore many midpoints by the Brunn-Minkowski inequality.

\begin{claim} \label{a forest of transportation}
For any graph $G$ and probability measures $\mu_A, \mu_B$, there exists an optimal transportation $\tau$ from $\mu_A$ to $\mu_B$ such that for all $A', B'$, we have that 
$$ |\{(a,b): a \in A', b \in B', \tau(a,b) > 0\} | \leq 2(|A'| + |B'|) - 1 . $$
\end{claim}
\begin{proof}
For a fixed transportation $\rho$, we define a bipartite graph $J_\rho$ with disjoint vertex sets $S_A \cup S_B$, where $S_A = \{v_a : \mu_A(a) > 0\}$ and $S_B = \{u_b : \mu_B(b) > 0\}$.
If $\mu_A(w) > 0$ and $\mu_B(w) > 0$ for some vertex $w$, then $v_w$ and $u_w$ both exist and are different.
We define the edges to be $E(J_\rho) = \{v_au_b : \rho(a, b) > 0\}$.
Among all optimal transportations, let $\tau$ be the one that minimizes $|E(J_\tau)|$.

We claim that $J_\tau$ is a forest.
By way of contradiction, suppose that a cycle $C = u_1,v_1,u_2,v_2,\ldots,u_k,v_k$ is in $J_\tau$.
Let $\epsilon_1 = \min_i \rho(u_i,v_i)$ and let $\epsilon_2 = \min_i \rho(u_{i+1}, v_i)$, where the indices are taken modulo $k$.
We consider two transportations $\tau', \tau''$, that equal $\tau$ on all pairs over vertices except
\begin{itemize}
	\item $\tau'(u_i, v_i) = \tau'(u_i, v_i) + \epsilon_2$, 
	\item $\tau'(u_{i+1}, v_i) = \tau'(u_{i+1}, v_i) - \epsilon_2$, 
	\item $\tau''(u_i, v_i) = \tau'(u_i, v_i) - \epsilon_1$,  and
	\item $\tau''(u_{i+1}, v_i) = \tau'(u_{i+1}, v_i) + \epsilon_1$. 
\end{itemize}
By cyclic monotonicity, we have that $\tau', \tau''$ are each optimal.
This is a contradiction, because by construction, $|E(J_{\tau'})|,|E(J_{\tau''})| < |E(J_\tau)|$.

Now consider vertex sets $A', B'$ and edge set $E' = \{(a,b): a \in A', b \in B', \tau(a,b) > 0\}$.
Each vertex in $A' \cup B'$ induces at most two vertices in $J_\tau$; let us call this induced subgraph $J[A' \cup B']$.
Each edge in $E'$ is in $J[A' \cup B']$.
Because $J_\tau$ is a forest, there are at most $|V(J[A' \cup B'])|-1$ such edges.
\end{proof}

The problem with Claim \ref{a forest of transportation} is that the Brunn-Minkowski inequality would collect many midpoints between $(a_i,b_j)$ where $\tau(a_i, b_j) = 0$.
Moreover, this seems like a fundamental flaw, given Example \ref{negative curvature}.
If anything, this is the complete opposite of the transportation that we did use.
By a proof similar to Theorem \ref{everybody is large}(2), it can be shown that if $\tau$ is optimal, maximizes entropy, and has no partition, then $\tau(a,b) > 0$ for all pairs with $d(a,b) = D$, $\mu_A(a) > 0$, $\mu_B(b) > 0$.
Regardless, this claim is an interesting statement in its own right.

The following statement turned out to be unnecessary for proving Theorem \ref{char strong curve}.
But it is true for general finite discrete spaces, and so it may be applicable towards a more general statement.

\begin{claim} \label{curvature over sets}
If $G$ has strong displacement convexity for probability measures $\mu_A, \mu_B$ that are uniform over vertex sets $A, B$, respectively, then $G$ has strong displacement convexity.
\end{claim}
\begin{proof}
Let $\mu_A, \mu_B$ be arbitrary probability distributions with optimal transport $\tau$ with geodesic choices $Q_{\alpha, \beta}$ which produces midpoint distribution $\mu_C$.
Suppose there exists some vertex $m$ such that $m = P_1(1/2) = P_2(1/2)$ for $Q_{\alpha_1, \beta_1}(P_1) > 0$, $Q_{\alpha_2, \beta_2}(P_2) > 0$, and $\mu_A(\alpha_1) > \mu_A(\alpha_2)$ (we allow the possibility of $\beta_1 = \beta_2$).
Let $\epsilon = \min\{\mu_A(\alpha_2), (\mu_A(\alpha_1) - \mu_A(\alpha_2))/2\} >0 $, and define $\mu_A'(\alpha_1) = \mu_A(\alpha+1) - \epsilon$, $\mu_A'(\alpha_2) = \mu_A(\alpha_2) + \epsilon$, and $\mu_A' = \mu_A$ otherwise.
By construction, there exists a transportation $\tau'$ from $\mu_A'$ to $\mu_B$ and geodesic choices $Q'_{\alpha, \beta}$ such that the midpoints have distribution $\mu_C$.
We propose that this is an optimal transposition from $\mu_A'$ to $\mu_B$.

Before proving the proposition, let us show how the proposition implies the claim.
By convexity we know that $S(\mu_A') > S(\mu_A)$, so if $\mu_A'$ to $\mu_B$ has curvature, then so does $\mu_A$ to $\mu_B$.
Repeat this process until $\mu_A(\alpha_1) = \mu_A(\alpha_2)$ whenever the geodesics of $\tau$ and $Q_{\alpha, \beta}$ involving $\alpha_1$ and $\alpha_2$ share a midpoint.
We define equivalency class $A_r = \{\alpha: \mu_A(\alpha) = r\}$ for $r>0$, and let $M_{A,r} = \sum_{\alpha \in A_r} \mu_A (\alpha)$ and $\mu_{A,r}$ be the probability distribution $\mu_A$ projected onto the set $A^r$ and scaled by $M_{A,r}^{-1}$.
If each transportation $\mu_{A,r} \otimes \mu_B$ using $\tau_{A,r}$ and $Q_{\alpha, \beta}$ has curvature, then the whole $\mu_A \otimes \mu_B$ will have curvature as the midpoints do not overlap.
But $\mu_{A,r}$ is a uniform measure over $A_r$, which is the conclusion of the claim.
Symmetrically we do this to $\mu_B$ as well, and so the claim holds.

So now we prove the proposition.
By Lemma \ref{separate into distances} $d(\alpha_1, \beta_1) = d(\alpha_2, \beta_2)$, and so $W^2(\tau') = W^2(\tau) = W^2(\mu_A, \mu_B)$.
Suppose there exists a transportation $\tau''$ from $\mu_A'$ to $\mu_B$ such that $W^2(\tau'') < W^2(\tau)$.
Among all such transportations, let $\tau''$ minimize $|\{(a,b): \tau''(a,b) \neq \tau'(a,b)\}|$.
If this number is $0$, then $W^2(\tau'') = W^2(\tau')$, which is a contradiction.

We consider the bipartite graph $J_{\tau'}$ and $J_{\tau''}$ as constructed in Claim \ref{a forest of transportation}.
There must exist some $u_1, v_1$ where $\tau'(u_1, v_1) < \tau''(u_1, v_1)$.
But both $\tau'$ and $\tau''$ move mass into the same distribution $\mu_B$, so there must exist a $u_2$ such that $\tau'(u_2, v_1) > \tau''(u_2, v_1)$.
We continue this process until we have found a cycle $u_1, v_2, \ldots, u_k, v_k$ such that for each $i$, $\tau'(u_i, v_i) < \tau''(u_i, v_i)$ and $\tau'(u_{i+1}, v_i) > \tau''(u_{i+1}, v_i)$.
We use this cycle to update $\tau'$ to be $\tau'_*$ or $\tau''$ to be $\tau''_*$ as in Claim \ref{a forest of transportation}, and one of two things happen: (1) $|\{(a,b): \tau''_*(a,b) \neq \tau'(a,b)\}|$ decreases and $W^2(\tau'') = W^2(\tau''_*)$, which is a contradiction, or (2) $W^2(\tau'_*) < W^2(\tau')$.
But notice that this update can also be done to $\tau$ into transportation $\tau_*$ with $W^2(\tau_*) < W^2(\tau)$, which contradicts the minimality of $\tau$.
\end{proof}

\end{document}